\documentclass[reqno]{amsart}%
\usepackage{palatino, mathpazo}
\usepackage{amsfonts}
\usepackage{amsmath}
\usepackage{amssymb,latexsym,xcolor}
\usepackage{graphicx}
\usepackage{harpoon}
\usepackage[mathscr]{eucal}
\usepackage{amssymb}%
\usepackage[%backref=page,
linktocpage=true,colorlinks,citecolor=magenta,linkcolor=blue,urlcolor=magenta]{hyperref}
\hypersetup{
 colorlinks=true,
 linkcolor=blue,
 citecolor=blue,
 urlcolor=blue,
 pdftitle={Uniqueness and nonuniqueness for mean field equations with multiple singularities on flat tori},
 pdfsubject={Uniqueness and nonuniqueness for mean field equations with multiple singularities on flat tori},
 pdfkeywords={Mean field equation, torus, uniqueness}
}

\providecommand{\U}[1]{\protect \rule{.1in}{.1in}}

\numberwithin{equation}{section}

\newcommand{\R}{\mathbb{R}}
\newcommand{\T}{\mathbb{T}}

\newcommand{\ud}{\mathrm d}

\newcommand{\Z}{\mathbb Z}
\newcommand{\C}{\mathbb C}
\renewcommand{\Cap}{\operatorname{Cap}}
\newcommand{\lam}{\lambda_1}
\theoremstyle{plain}
\newtheorem{theorem}{Theorem}[section]

\newtheorem{lemma}[theorem]{Lemma}
\newtheorem{corollary}[theorem]{Corollary}

\theoremstyle{remark}
\newtheorem{remark}[theorem]{Remark}

\usepackage{tikz}
\usetikzlibrary{arrows.meta,decorations.pathreplacing,calc}

\tikzset{
	annbox/.style={line width=0.45pt},
	annid/.style={-{Stealth[length=2mm]},line width=0.45pt},
	annnodal/.style={line width=1.25pt},
	annshade/.style={fill=black!8,draw=none},
	annsmall/.style={font=\scriptsize,align=center},
	anntitle/.style={font=\footnotesize\bfseries,align=center},
	annmark/.style={circle,fill=black,inner sep=1.6pt}
}

\newcommand{\AnnTorusFrame}{%
	\draw[annbox] (0,0) rectangle (4,2.8);
	\draw[annid] (1.45,-0.15)--(2.55,-0.15);
	\draw[annid] (1.45,2.95)--(2.55,2.95);
	\draw[annid] (-0.15,0.85)--(-0.15,1.95);
	\draw[annid] (4.15,0.85)--(4.15,1.95);
}

\newcommand{\AnnCaseVertex}{%
	\begin{tikzpicture}[scale=0.9,every node/.style={transform shape}]
		\AnnTorusFrame
		\fill[annshade] (0,0) rectangle (2,1.4);
		\draw[annnodal] (2,0)--(2,2.8);
		\draw[annnodal] (0,1.4)--(4,1.4);
		\node[annmark] at (2,1.4) {};
		\node[anntitle] at (2,3.35) {vertex};
		\node[annsmall] at (2.75,1.75) {degree $\ge4$};
		\node[annsmall] at (1,0.70) {a disk face\\is forced};
		\node[annsmall] at (2,-0.47) {opposite sides are identified};
	\end{tikzpicture}
}

\newcommand{\AnnCaseParallel}{%
	\begin{tikzpicture}[scale=0.9,every node/.style={transform shape}]
		\AnnTorusFrame
		\fill[black!7] (0,0) rectangle (0.75,2.8);
		\fill[black!7] (1.55,0) rectangle (2.35,2.8);
		\fill[black!7] (3.15,0) rectangle (4,2.8);
		
		\draw[annnodal] (0.75,0)--(0.75,2.8);
		\draw[annnodal] (1.55,0)--(1.55,2.8);
		\draw[annnodal] (2.35,0)--(2.35,2.8);
		\draw[annnodal] (3.15,0)--(3.15,2.8);
		
		\node[anntitle] at (2,3.35) {parallel essential loops};
		\node[annsmall] at (2.75,1.40) {annulus};
		\node[annsmall] at (0.37,1.40) {same};
		\node[annsmall] at (2,-0.47) {opposite sides are identified};
	\end{tikzpicture}
}

\begin{document}

	\title[Uniqueness and Non-uniqueness]{Uniqueness and nonuniqueness for mean field equations with multiple singularities on flat tori}

\author{Zhijie Chen}
\address{Department of Mathematical Sciences, Yau Mathematical Sciences Center,
Tsinghua University, Beijing, 100084, China }
\email{zjchen2016@tsinghua.edu.cn}
\author{Shihong Zhang}
\address{Yau Mathematical Sciences Center,
Tsinghua University, Beijing, 100084, China}
\email{shihong-zhang@tsinghua.edu.cn}

%\date{August 2026}

	\begin{abstract}
We study mean field equations with multiple positive singularities on flat tori. We establish a uniform uniqueness criterion in terms of the total singular mass and the scale-invariant spectral quantity $\lambda_1(\T)|\T|$, where $\lambda_1(\T)$ denotes the first positive eigenvalue of the Laplacian. Combining nodal-set analysis for Jacobi fields with a cylindrical Alexandrov--Bol type inequality and sharp small-capacity asymptotics, we obtain the optimal leading coefficient $8/\pi^2$ in the total-mass uniqueness bound as $\lambda_1(\T)|\T|\to0$. Sharpness is demonstrated by examples with two symmetric singularities on degenerating rectangular tori that admit at least three distinct solutions. For the case of two singularities with total mass $2\rho$, we further derive uniqueness and multiplicity results near $\rho=4\pi$ from the critical-point structure of the associated symmetrized Green function. In particular, when its only critical points are the four nondegenerate two-torsion points, the equation admits a unique solution for $\rho$ just below $4\pi$ and exactly three solutions for $\rho$ just above $4\pi$.
\end{abstract}
	
\maketitle

% This section requires amsmath, amssymb, and amsthm.
% The citation keys should be replaced by the keys used in your .bib file.

\section{Introduction}
\label{sec 1}

Mean field equations of Liouville type arise naturally in conformal geometry,
statistical mechanics, and gauge field theory. From the geometric point of
view, the exponential nonlinearity describes the conformal factor of a metric
with prescribed Gaussian curvature, and the resulting mean field equations on
compact surfaces have been extensively studied; see, for instance,
\cite{BGJM, BKLY,BartolucciLin, BT, CM, CL-1,CL-2,CL-3, DKM, DingJostLiWang99, LiYY, LT, MR, WeiZhang2026, WWX} and the references therein. In statistical mechanics, the mean field equation
arises in the Onsager theory of two-dimensional point vortices, where it
describes the equilibrium distribution of a large system of interacting
vortices; see \cite{CLMP}. Singular Liouville equations also appear naturally
in Abelian Chern--Simons--Higgs theory and related gauge field models, where
the prescribed singular points correspond to vortices; see, for example,
\cite{Tarantello08}. These geometric and physical connections have motivated
an extensive study of existence, blow-up, symmetry, non-degeneracy, and
uniqueness for Liouville-type equations on compact surfaces.

The natural smooth version on a flat torus $\mathbb{T}$ is the
normalized mean field equation
\begin{equation}\label{eq:regular-mean-field}
    \Delta u
    +\rho\left(
        \frac{e^u}
        {\displaystyle\int_{\mathbb{T}}e^u\,dA}
        -\frac{1}{|\mathbb{T}|}
    \right)=0,\qquad\mathrm{on}\qquad \mathbb{T}.
\end{equation}
Every constant function solves \eqref{eq:regular-mean-field}, and a
fundamental question is to determine when these are the only solutions.
Important uniqueness and one-dimensional symmetry results on rectangular
tori were obtained by Lin--Lucia \cite{LinLucia06,LinLucia07} and Gui--Moradifam \cite{GuiMoradifam19}. A uniform result for
arbitrary flat tori was subsequently established by Gu--Gui--Hu--Li \cite{GGHL}, who
proved that \eqref{eq:regular-mean-field} admits only constant solutions
whenever
\begin{equation}\label{eq:regular-uniqueness-threshold}
    \rho
    \leq
    \min\left\{
        8\pi,\,
        \lambda_1(\mathbb{T})|\mathbb{T}|
    \right\}.
\end{equation}
Here $\lambda_1(\mathbb T)$ is the first positive eigenvalue of $-\Delta$ on $\mathbb T$. Thus, even in the absence of singularity, the uniqueness is governed
jointly by the critical mass $8\pi$ from two-dimensional conformal analysis
and the spectral geometry of the underlying torus.

A substantially different phenomenon appears when one singularity is
introduced. After a translation, the equation can be written as
\begin{equation}\label{eq:one-source-intro}
    \Delta u+e^u
    =
    \rho\,\delta_0
    \qquad\text{on }
    \mathbb T=E_\tau
    :=
    \mathbb{C}/(\mathbb{Z}+\mathbb{Z}\tau),
\end{equation}
where $\tau\in\mathbb C$ satisfies $\mathrm{Im}\tau>0$.
At the quantized values
\[
    \rho=8\pi n,
    \qquad n\in\mathbb{N},
\]
the solvability of \eqref{eq:one-source-intro} may depend delicately on the
conformal modulus $\tau$. For $n=1$, Lin and Wang \cite{LW5} related even solutions of
\eqref{eq:one-source-intro} to nontrivial critical points of the Green
function and proved that the Green function of a flat torus has either three
or five critical points. For general $n$, the developing-map formulation
connects the singular Liouville equation with the classical Lam\'e equation,
hyperelliptic curves, and modular forms. This program was developed further
by Chai--Lin--Wang and their collaborators; see
\cite{ChaiLinWang15,ChenKuoLin19, LW5,LinWang17}.
These results demonstrate that the one singularity problem possesses a remarkable
algebraic structure which is special to the quantized singular equation. On the other hand, concerning the uniqueness, Lin and Wang \cite[Theorem 1.4]{LinWang17} proved that the solution of \eqref{eq:one-source-intro} is unique whenever 
\begin{equation}\label{eq:one-rho}
 \rho\in (0,8\pi).   
\end{equation}
Remark that the range \eqref{eq:one-rho} is universal for arbitrary torus, which is quite different from \eqref{eq:regular-uniqueness-threshold}.
A natural question is whether such kinds of universal uniqueness results \eqref{eq:one-rho} hold or not when the equation has multiple singularities. 

In the present paper, we study this uniqueness problem for multiple singularities on an
arbitrary flat torus, and give a negative answer to this universal problem. More precisely, let
$
    \Lambda\subset\mathbb{C}
$
be a lattice and write
$
    \mathbb{T}:=\mathbb{C}/\Lambda.
$
Let $p_1,\ldots,p_N\in\mathbb{T}$ be pairwise distinct, and let
$\beta_1,\ldots,\beta_N>0$, where $N\geq 2$. We consider distributional solutions of the mean field equation with multiple singular sources
\begin{equation}\label{eq:main}
    \Delta u+e^u
    =
    \sum_{j=1}^{N}\beta_j\delta_{p_j}
    \qquad\text{on }\;\mathbb{T},
    \qquad
    e^u\in L^1(\mathbb{T}).
\end{equation}
The regularity theory \cite{BM} shows that any solution \eqref{eq:main} satisfies $u\in C^2(\mathbb T\setminus\{p_j\}_{j=1}^N)$ and
$$u(z)=\frac{\beta_j}{2\pi}\log|z-p_j|+O(1),\quad\text{near }z=p_j.$$
Integrating \eqref{eq:main} over $\mathbb{T}$ gives
\begin{equation}\label{eq:def-rho}
   \int_{\mathbb{T}}e^u\,dA
    =
    \sum_{j=1}^{N}\beta_j:=N\rho.
\end{equation}
It is well known that by applying the standard variational approach and using the Moser-Trudinger inequality, equation \eqref{eq:main} always has solutions whenever $0<\rho<8\pi/N$.
However, in contrast to the case $N=1$ in \eqref{eq:one-source-intro}, we will prove in this paper that, one can not expect the existence of a universal constant $\rho_{N}$ independent of the choice of torus (say $4\pi$ or $2\pi$ for example) to guarantee the uniqueness for arbitrary torus when $\rho\in (0,\frac1N\rho_{N})$.

Before introducing our main theorem, we first discuss the quantity $\lambda_1(\T)|\T|$.
 Recall that, for a flat torus $\T=\mathbb{R}^2/\Lambda$, the dual lattice is
\begin{align}\label{Intro lambda}
    \Lambda^*:=\left\{\xi\in\mathbb{R}^2:\xi\cdot\ell\in2\pi\mathbb{Z}
\ \text{for every }\ell\in\Lambda\right\}.
\end{align}
Since $e^{i\xi\cdot x}$ is an eigenfunction of $-\Delta$ with eigenvalue $|\xi|^2$ for every $\xi\in\Lambda^*$, the first positive eigenvalue is
\begin{align}\label{Intro lambda_1}
\lambda_1(\T)=\min_{\xi\in\Lambda^*\setminus\{0\}}|\xi|^2.
\end{align}
The covolume of $\Lambda^*$ is $(2\pi)^2/|\T|$. By the two-dimensional Hermite inequality (cf. \cite{Cassels1997,ConwaySloane1999}),
\[
\min_{\xi\in\Lambda^*\setminus\{0\}}|\xi|^2
\leq \frac{2}{\sqrt{3}}\,\operatorname{covol}(\Lambda^*)
=\frac{8\pi^2}{\sqrt{3}\,|\T|}.
\]
Therefore
\[
0<\lambda_1(\T)|\T|\leq\frac{8\pi^2}{\sqrt{3}}.
\]
Moreover, $\lambda_1(\T)|\T|$ is invariant under homothetic rescaling and therefore depends only on the shape of the flat torus. It can also tend to $0$ as the torus degenerates; see Section \ref{Sec 5}.

Let $j_{0,1}=2.4048\cdots$ be the first positive zero of the Bessel function
$J_0$, and set \begin{equation}\label{kappa}\kappa:=j_{0,1}/2=1.2024\cdots.\end{equation}
Our first main result gives a uniform uniqueness criterion depending only on
the total singular mass and the spectral geometry of $\mathbb{T}$.

\begin{theorem}\label{thm:main}
Let $N\geq 2$. If $0<\rho \leq\frac{1}{N}\min\left\{4\pi,\frac{8\pi \kappa\lambda_1(\T)|\T|}{4\pi^2+\kappa\lambda_1(\T)|\T| } \right\}$,
	then equation \eqref{eq:main} has a unique solution.
\end{theorem}

No symmetry or special arithmetic condition is imposed on the singular
points $p_j$. The proof combines an Alexandrov--Bol type inequality on
cylinders with a nodal set reduction adapted to the topology of the torus.
A central role is played by the Jacobi equation
\begin{equation}\label{eq:jacobi-intro}
\Delta\phi+e^u\phi=0
\qquad\text{on }\mathbb{T}.
\end{equation}
For the Jacobi fields considered in this paper, testing
\eqref{eq:jacobi-intro} against the constant function $1$ gives
\begin{equation}\label{eq:jacobi-orthogonality-intro-intro-intro}
\int_{\mathbb{T}}e^u\phi dA=0.
\end{equation}
Thus, a nontrivial Jacobi field must change sign. When $\rho \leq \frac{4\pi}{N}$, the topology of its nodal set allows the torus to be decomposed into suitable cylindrical domains. Here, we emphasize that our approach is essentially different from the traditional nodal set analysis in \cite{GGHL,GuiMoradifam,GuiMoradifam19,LinLucia06,LinLucia07,LW5}. These classical approaches rely on two key points: first, for any simply connected nodal domain $\Omega$, the Alexandrov--Bol type inequality yields $\int_{\Omega}e^u\geq 4\pi$; second, one obtains a lower bound on the number of simply connected nodal domains and then derives a contradiction. In contrast, on a flat torus, we prove that the nodal domains of the Jacobi field are topological cylinders, which are not simply connected. To overcome this difficulty, we extend the Alexandrov--Bol type inequality from simply connected domains to cylindrical domains.

Our next result improves the range of $\rho$ in Theorem \ref{thm:main} to an almost sharp range when $\lam(\T)|\T|$ is small. The almost sharpness will be explained in Remark \ref{rmk}.

\begin{theorem}
\label{thm:intro-asymptotic-uniqueness}
For every $\varepsilon\in(0,1)$, there exists
$\delta_\varepsilon>0$ such that, for every flat torus
$\T$ and every $N\geq2$, equation~\eqref{eq:main} has a unique
solution whenever
\[
 0<\lam(\T)|\T|\leq\delta_\varepsilon,
 \qquad
 0<\rho\leq
 \frac{8(1-\varepsilon)}{N\pi^2}\lam(\T)|\T|.
\]
The constant $\delta_\varepsilon$ is independent of $N$,
the  distinct singular points $p_1,\ldots,p_N$,
and the constants $\beta_1,\ldots,\beta_N$.
\end{theorem}

To illustrate both the geometric necessity and the limitations of the
uniform bound in Theorem~\ref{thm:main}, we next consider the equation with
two symmetric singular points on a normalized flat torus. Let $\tau \in \mathbb{H}=\left \{  \tau\in\mathbb C|\operatorname{Im}\tau>0\right \}$, $\Lambda_{\tau}=\mathbb{Z}+\mathbb{Z}\tau$, and denote
$$\omega_{0}=0,\quad\omega_{1}=1,\quad\omega_{2}=\tau,\quad\omega_{3}=1+\tau.$$Let $\T= E_{\tau}:=\mathbb{C}/\Lambda_{\tau}$ be a normalized flat torus and $E_{\tau}[2]:=\{ \frac{\omega_{k}}{2}|k=0,1,2,3\}+\Lambda
_{\tau}$ be the set consisting of the lattice points and half periods
in $E_{\tau}$.  
We study the equation
\begin{equation}\label{eq:two-source}
    \Delta u+e^u
    =
    \rho\bigl(\delta_p+\delta_{-p}\bigr)
    \qquad\text{on }E_\tau.
\end{equation}
Here the total singular mass is $2\rho$.

Our first result for \eqref{eq:two-source} shows that non-uniqueness may
already occur at a mass scale determined by
$\lambda_1(E_\tau)|E_\tau|$.

\begin{theorem}
\label{thm:intro-nonuniqueness}
Let $p=\tau/4\in E_\tau$ and
$\sigma_\tau(z)=z+\tau/2$. The following statements hold
for equation~\eqref{eq:two-source}.
\begin{enumerate}
 \item[(i)]
 If $\mathrm{Im}\,\tau>\pi/2$ and
 $
  \frac12\lam(E_\tau)|E_\tau|<\rho<4\pi,
 $
 then the equation admits at least two distinct solutions
 that are not $\sigma_\tau$-invariant and exactly one unique
 $\sigma_\tau$-invariant solution.

 \item[(ii)]
 For every $\varepsilon>0$, there exists $b_\varepsilon\geq1$
 such that, whenever $\tau=\mathrm{i}b$ with
 $b\geq b_\varepsilon$ and
 \[
  \rho=\frac{16(1+\varepsilon)}{b}
  =\frac{4(1+\varepsilon)}{\pi^2}
   \lam(E_\tau)|E_\tau|,
 \]
 the equation admits at least  two distinct solutions
 that are not $\sigma_\tau$-invariant and exactly one unique
 $\sigma_\tau$-invariant solution.
\end{enumerate}
\end{theorem}

Note that $\frac12\lam(E_\tau)|E_\tau|<4\pi$ if $\mathrm{Im}\,\tau>\pi/2$; see Section \ref{Sec 6}.
Theorem \ref{thm:intro-nonuniqueness} explains why a uniqueness criterion valid uniformly over all
flat tori must involve the spectral quantity
$\lambda_1(\mathbb{T})|\mathbb{T}|$. More precisely, if
\[
    L_\tau
    :=
    \lambda_1(E_\tau)|E_\tau|,
\]
then Theorem~\ref{thm:main}, specialized to $N=2$, gives uniqueness for
\[
    \rho
    \leq
    \frac{4\pi \kappa L_\tau}{4\pi^2+\kappa L_\tau}
    =
    \frac{\kappa L_\tau}{\pi}
    +O(L_\tau^2)
    \qquad\text{as }L_\tau\to0,
\]
whereas Theorem~\ref{thm:intro-nonuniqueness} gives non-uniqueness for
$
    \rho>\frac{L_\tau}{2}
$
for the above family of tori and singular points. Thus, in the degenerating
regime, the linear dependence on
$\lambda_1(\mathbb{T})|\mathbb{T}|$ has the correct order, although these two
results do not assert that the numerical constant in
Theorem~\ref{thm:main} is optimal.

\begin{remark}\label{rmk}
In Theorem \ref{thm:intro-nonuniqueness} (ii), we have
$$\frac{4(1+\varepsilon)}{\pi^2}
   \lam(E_\tau)|E_\tau|\searrow \frac{4}{\pi^2}
   \lam(E_\tau)|E_\tau|\quad\text{as }\varepsilon\to 0,$$
   while in Theorem \ref{thm:intro-asymptotic-uniqueness} with $N=2$,
   $$\frac{8(1-\varepsilon)}{N\pi^2}\lam(\T)|\T|\nearrow \frac{4}{\pi^2}\lam(\T)|\T|\quad\text{as }\varepsilon\to 0.$$
   This indicates that the range $0<\rho\leq
 \frac{8(1-\varepsilon)}{N\pi^2}\lam(\T)|\T|$ in Theorem \ref{thm:intro-asymptotic-uniqueness} is almost optimal at least for $N=2$.
\end{remark}

On the other hand, the solution structure of \eqref{eq:two-source} is not
controlled by the spectral quantity alone. It may depend sensitively on both
the conformal modulus $\tau$ and the position of the singular point $p$. To
formulate a complementary near-critical uniqueness result, let
$G(z)=G(z;\tau)$ be the normalized Green function of $E_\tau$, defined by
\[
    -\Delta G
    =
    \delta_0-\frac{1}{|E_\tau|},
    \qquad
    \int_{E_\tau}G\,dA=0,
\]
and define
\begin{equation}\label{eq:def-Gp}
    G_p(z)
    :=
    \frac{1}{2}
    \bigl(
        G(z+p)+G(z-p)
    \bigr),\quad\text{for }p\notin E_{\tau}[2].
\end{equation}
The four points in $E_\tau[2]$ are always critical points of $G_p$ and will
be called the trivial critical points. A critical point outside
$E_\tau[2]$ will be called nontrivial.

\begin{theorem}\label{thm 1.5}
Fix $\tau\in\mathbb H$ and $p\in E_\tau\setminus E_\tau[2]$.
Assume that the only critical points of $G_p$ are the four points in
$E_\tau[2]$, and that all four are nondegenerate.
Then there exists $\varepsilon=\varepsilon(\tau,p)>0$ such that,
for every $\rho\in(4\pi-\varepsilon,4\pi)$, equation~\eqref{eq:two-source}
admits a unique solution, which is even.
\end{theorem}

\begin{remark}
The existence of pairs $(\tau,p)$ satisfying these assumptions follows
from \cite{CFL}.
Under the same assumptions, equation \eqref{eq:two-source} has exactly three solutions,
all even, for $\rho\in(4\pi,4\pi+\varepsilon)$.
This conclusion, together with further multiplicity results for other
critical-point configurations of $G_p$, is proved in Section \ref{Sec 6}.
\end{remark}

These results distinguish uniform uniqueness estimates from the finer
behavior of one singularity equation.
Theorems \ref{thm:main} and \ref{thm:intro-asymptotic-uniqueness} provide uniqueness criteria that are uniform in the
number, locations, and strengths of the singularities.
The two-singularity examples  in Theorem \ref{thm:intro-nonuniqueness} show that nonuniqueness can occur
at arbitrarily small total mass as the torus degenerates; moreover,
the rectangular examples establish the asymptotic optimality of
Theorem~\ref{thm:intro-asymptotic-uniqueness} already for two singularities.
In contrast, Theorem~\ref{thm 1.5} gives uniqueness in a left neighborhood of
$\rho=4\pi$ under the stated assumptions on $G_p$.
Together with the results of Section~\ref{Sec 6}, this shows that the spectral
quantity $\lambda_1(\T)|\T|$ governs the uniform small-mass uniqueness
scale, while the critical-point structure of $G_p$ provides more detailed
information on uniqueness and multiplicity near the critical parameter.

The organization of this paper is as follows.  In Section \ref{Sec 2}, we establish uniqueness for sufficiently small parameters, which provides the starting point for the subsequent continuation argument. Section \ref{Sec 3} studies the nodal structure of Jacobi fields and shows, under the relevant mass bound, that their nodal domains must be parallel annuli. In Section \ref{Sec 4}, we prove an Alexandrov–Bol type inequality on cylinders and derive its sharp small-capacity asymptotics. These analytic and topological estimates are combined in Section \ref{Sec 5} with the spectral geometry of flat tori to prove the uniform uniqueness results, Theorems \ref{thm:main} and \ref{thm:intro-asymptotic-uniqueness}. Finally, Section \ref{Sec 6} constructs non-unique solutions by variational arguments and analyzes the solution structure near the critical parameter \(4\pi\) through the critical points of the Green function and blow-up theory, thereby proving Theorems \ref{thm:intro-nonuniqueness} and \ref{thm 1.5} as well as the accompanying multiplicity results.

\section{Start up: uniqueness for small parameter}\label{Sec 2}
To prove uniqueness, our main strategy is a continuity argument, which consists of two steps:
\begin{itemize}
    \item For sufficiently small $\rho>0$, equation \eqref{Sec 2 equ-1} has a unique solution.
    \item There exists a threshold $\tilde{\rho}>0$ such that, for every $\rho\in(0,\tilde{\rho}]$, the linearized operator at any solution is non-degenerate. Equivalently, if $u$ is a solution of \eqref{Sec 2 equ-1} and $\phi$ satisfies
    \[
        \Delta \phi+e^u\phi=0
        \qquad \text{in }\T,
    \]
    then $\phi\equiv0$.
\end{itemize}
Combining this with the compactness results in \cite{LiYY}, we obtain the uniqueness of solutions to \eqref{eq:main}.
In this section, we establish Step~1.

\begin{lemma}Let $p_1,\ldots,p_N\in\T$ be distinct points and let
$\beta_1,\ldots,\beta_N>0$. Set
$
    N\rho:=\sum_{j=1}^N\beta_j.
$
Then there exists a small
$
    \rho_0=\rho_0(\T,p_1,\ldots,p_N)\in (0, 8\pi/N)
$
such that, whenever $\rho\in(0,\rho_0)$, the equation
\begin{align}\label{Sec 2 equ-1}
    \Delta u+e^u
    =\sum_{j=1}^N\beta_j\delta_{p_j}
    \qquad\text{in }\T
\end{align}
admits a unique solution.
\end{lemma}

\begin{proof}
	Throughout the proof, $C$ denotes a positive constant depending only on
	$\T$ and $p_1,\ldots,p_N$. Let $G$ be the Green function on $\T$ normalized by
	\[
	-\Delta_xG(x,q)=\delta_q-\frac1{|\T|},
	\qquad
	\int_{\T}G(x,q)\,\ud A_x=0 .
	\]
	Set
	\[
	S(x)=-\sum_{j=1}^N\beta_jG(x,p_j),
	\qquad
	h=e^S,
	\]
	then
	\[
	\Delta S=\sum_{j=1}^N\beta_j\delta_{p_j}-\frac {N\rho}{|\T|}.
	\]
	Moreover, since $\beta_j>0$ for $1\leq j\leq N$,  if $\rho \le 1$, then
	$
	0\le h\le C .
	$

    Let $0<\rho<\min\{1, 8\pi/N\}$. Then it was pointed out in Section \ref{sec 1} that \eqref{Sec 2 equ-1} always has solutions.
	Let $u$ be a solution. Since integration gives
	\[
	\int_{\T}e^u\,\ud A=N\rho,
	\]
	we write
	\[
	u=S+v+c,
	\qquad
	\int_{\T}v\,\ud A=0 .
	\]
	Then
	\begin{align}\label{Small lem equ-a0}
	-\Delta v
	=
	N\rho\left(
	\frac{he^v}{\int_{\T}he^v\,\ud A}
	-\frac1{|\T|}
	\right).
	\end{align}
By the Green
	representation formula, we can get the integral representation
	\[
	v(x)=N\rho \int_{\T}G(x,y)
	\left(P(v)(y)-\frac1{|\T|}\right)\ud A_y,
	\]
	where 
	\begin{align*}
	P(v)=\frac{he^v}{\int_{\T}he^v\,\ud A}.
	\end{align*}
 Then, by Jensen's inequality and \begin{align*}
 	\left|G(x,y)-\frac{1}{2\pi}\log\frac{1}{|x-y|}\right|\leq C,
 \end{align*}for $\rho>0 $ sufficiently small, we can estimate
	\begin{align}\label{Small Lem equ-a}
	\int_{\T}e^{2|v(x)|}\,\ud A_x\leq &\int_{\T}\exp\left(\int_{\T}4N\rho|G(x,y)|\frac{P(v)(y)+\frac{1}{|\T|}}{2}\ud A_y\right)\ud A_x \nonumber\\
    \leq & \int_{\T}\int_{\T}e^{4N\rho|G(x,y)|}\frac{P(v)(y)+\frac{1}{|\T|}}{2}\ud A_y\ud A_x\nonumber\\    
	\leq& C\int_{\T} \frac{P(v)(y)+\frac{1}{|\T|}}{2}\int_{\T}\frac{1}{|x-y|^{\frac{2\rho N}{\pi}}}\ud A_x \ud A_y\leq C.
	\end{align}
	
	Next, since
	\[
	\int_{\T}S\,\ud A=\int_{\T}v\,\ud A=0,
	\]
	Jensen's inequality gives
	\begin{align}\label{Small Lem equ-b}
	\int_{\T}he^v\,\ud A
	=
	\int_{\T}e^{S+v}\,\ud A
	\ge |\T|.                                 
	\end{align}
	Combining $h\le C$, \eqref{Small Lem equ-a} and \eqref{Small Lem equ-b}, we obtain
	\[
	\|P(v)\|_{L^2(\T)}\le C .
	\]
	Therefore, from \eqref{Small lem equ-a0},
	\[
	\|v\|_{W^{2,2}(\T)}\le C\rho .          
	\]
	Since $W^{2,2}(\T)\hookrightarrow C^0(\T)$, after decreasing $\rho_0$ if
	necessary, every solution with $0<\rho<\rho_0$ satisfies
	\[
	\|v\|_{C^0(\T)}\le 1 .                  
	\]
	
	Now let $v_1$ and $v_2$ be two solutions of \eqref{Small lem equ-a0}, and set
	\[
	w=v_1-v_2 .
	\]
	Then
	\begin{align}\label{Small lem equ-c0}
	-\Delta w=N\rho\bigl(P(v_1)-P(v_2)\bigr),
	\qquad
	\int_{\T}w\,\ud A=0 .                
	\end{align}
	We claim that
	\begin{align}\label{Small lem equ-c}
	\|P(v_1)-P(v_2)\|_{L^2(\T)}
	\le C\|v_1-v_2\|_{L^2(\T)} .           
	\end{align}
	Indeed,
	\[
	P(v_1)-P(v_2)
	=
	\frac{h(e^{v_1}-e^{v_2})}{\int_{\T}he^{v_1}}
	+
	he^{v_2}
	\frac{\int_{\T}h(e^{v_2}-e^{v_1})}
	{\left(\int_{\T}he^{v_1}\right)
		\left(\int_{\T}he^{v_2}\right)} .
	\]
	Using \eqref{Small Lem equ-b}, $h\le C$, and $\|v_i\|_{C^0}\le1$, this immediately gives
	\eqref{Small lem equ-c}.
	
	Applying the elliptic estimate to \eqref{Small lem equ-c0}, we get
	\[
	\|w\|_{W^{2,2}}
	\le
	C\rho\|P(v_1)-P(v_2)\|_{L^2}
	\le
	C\rho\|w\|_{L^2}
	\le
	C\rho\|w\|_{W^{2,2}} .
	\]
	Choose $\rho_0>0$ so small that $C\rho_0<1$. Then for any $0<\rho<\rho_0$ we have $w\equiv0$, hence
	$v_1=v_2$. Since
	\[
	e^c=\frac{N\rho}{\int_{\T}he^v\,\ud A},
	\]
	the constant $c$ is also uniquely determined. Therefore $u_1=u_2$.
\end{proof}
In the following two sections, we focus on finding a specific number $\tilde \rho$ such that the Jacobi field $\phi\equiv 0$ for $\rho\in (0,\tilde \rho]$; namely, the PDE
$$
        \Delta \phi+e^u\phi=0
        \qquad \text{in }\T, \quad \rho\in (0,\tilde \rho]
$$
has only the trivial solution.

\section{Nodal set analysis}\label{Sec 3}
In this section, we apply an Alexandrov–Bol type inequality to analyze the nodal sets of solutions on the torus $\T$. To this end, we first recall the following Bol's inequality:

\begin{lemma}[{\cite[Lemma 4.2]{LW5}}]
	\label{lem:weighted-Bol Ineq}
	Let \(\Omega\subset\mathbb R^2\) be a bounded simply connected domain whose boundary is
	piecewise \(C^2\). Assume that, in the sense of distributions,
	\begin{equation}\label{eq:positive-singular-Liouville}
		\Delta U+e^U
		=
		\sum_{j=1}^{N}\beta_j\delta_{q_j}
		\qquad\text{in } \Omega,
		\qquad
		\beta_j\ge0,\quad q_j\in \overline{\Omega},
	\end{equation}
	and assume the regular part of $U$ is $C^2(\Omega)$ and $\int_\Omega e^U\leq 8\pi$, then for any smooth subdomain $\Omega^{'}\subset \Omega$ it holds
	\begin{align}\label{Bol Inequ}
		\left(\int_{\partial \Omega^{'}}e^{U/2}\right)^2\geq \frac{1}{2}\left(\int_{ \Omega^{'}}e^{U}\right)\left(8\pi-\int_{\Omega^{'}}e^{U} \right).
	\end{align}
\end{lemma}
\begin{remark}
Lemma~\ref{lem:weighted-Bol Ineq} also holds when some $q_j\in\partial\Omega' $.
Indeed, one can approximate $\Omega'$ from the interior by smooth 
 domains avoiding the boundary singularities and then pass to the
limit. The asymptotic behavior
$
U(x)=\frac{\beta_j}{2\pi}\log|x-q_j|+O(1), \beta_j\geq0,
$
ensures that the contributions of the artificial boundary arcs vanish in
the limit. We refer to
\cite{BartolucciCastorina2019,BartolucciJevnikarLin2019}
for related approximation arguments for singular Alexandrov--Bol
inequalities.
\end{remark}

Using  the Bol's inequality \eqref{Bol Inequ}, we can derive the following weighted eigenvalue estimate.

\begin{lemma}
	\label{lem:weighted-spectral-positive-singularities}
	Assume that $U$ is a solution of \eqref{eq:positive-singular-Liouville} and 
set
	\[
	\sigma:=\int_\Omega e^U\,dx\le4\pi.
	\]
Recall $\kappa>1$ defined in \eqref{kappa}.	Then, for every \(v\in H_0^1(\Omega)\),
	\begin{equation}\label{eeffc}
\int_\Omega|\nabla v|^2\,dx\geq\left(\frac{4\pi\kappa^2}{\sigma}+1-\kappa^2\right)		\int_\Omega e^U v^2\,dx.
	\end{equation}
\end{lemma}

\begin{proof}
By density, it is enough to first consider \(v\in C_c^\infty(\Omega)\), \(v\ge0\).
For \(t\ge0\), define
\[
	\Omega_t:=\{x\in \Omega : v(x)>t\},\qquad
	A(t):=\int_{\Omega_t}e^U\,dx.
	\]
For almost every regular value \(t\), the boundary \(\partial \Omega_t\) is a
finite union of smooth Jordan curves.

We compare \(v\) with a radial rearrangement on the standard spherical
model. Define
\[
U_0(x)
:=
2\log\left(\frac{8}{8+|x|^2}\right),
\qquad
e^{U_0(x)}
=
\frac{64}{(8+|x|^2)^2}.
\]
Then
\[
\Delta U_0+e^{U_0}=0
\qquad\text{in }\mathbb R^2.
\]
First, we define the rearrangement of $\Omega_t$ with respect to the measures $e^{U_0}dx$ and $e^Udx$ by
\begin{align*}
    \int_{\Omega_t^*}e^{U_0}=\int_{\Omega_t}e^U,
\end{align*}
where $\Omega_t^*=B_{R(t)}:=\{x\in\mathbb R^2 : |x|<R(t)\}$ and $R(t)\leq R_{\sigma}\leq 2\sqrt{2}$. Here,
\begin{equation}\label{Rsigma}
\int_{B_{R_{\sigma}}}e^{U_0}\,dx=\frac{8\pi R_{\sigma}^2}{8+R_{\sigma}^2} =\sigma.
\end{equation}

Then, we define the rearrangement \(v^*\) of $v$ to be the radial decreasing function on \(B_{R_{\sigma}}\) defined
by
\[
	\{v^*>t\}=\{v>t\}^*=\Omega_t^*.
	\]
Thus \(v^*\) is equimeasurable with \(v\) with respect to the two measures
\(e^Udx\) and \(e^{U_0}dx\). In particular, by the layer-cake formula,
\begin{equation}\label{eq:L2-equimeasurable}
\int_{\Omega} e^Uv^2dx
=
\int_{B_{R_{\sigma}}}e^{U_0}(v^*)^2dx.
\end{equation}
By the coarea formula and Bol's inequality \eqref{Bol Inequ}, we obtain
\begin{equation}\label{eq:PS-weighted}
\int_{B_{R_{\sigma}}}|\nabla v^*|^2dx
\le
\int_{\Omega}|\nabla v|^2dx.
\end{equation}
Since \eqref{eq:PS-weighted} is standard to experts in this field, we omit the details. For further details, we refer the reader to \cite{BartolucciCastorina2019,BGJM,BartolucciJevnikarLin2019,GuiMoradifam}.

It remains to estimate the weighted first Dirichlet eigenvalue of the
model ball \(B_{R_{\sigma}}\).
 Recalling \eqref{Rsigma},
we let $f=f(r)$ be a smooth radial function on $\overline{B_{R_\sigma}}$
with $f(R_{\sigma})=0$, and set
\[
s:=\frac{1}{\sigma}\int_{B_r}e^{U_0}\,dx
=\frac{8\pi r^2}{\sigma(8+r^2)},
\qquad F(s):=f(r),\qquad 0\le s\le1.
\]
Then $F(1)=0$, and a change of variables gives
\begin{equation}\label{eq:model-barta}
\begin{aligned}
\int_{B_{R_\sigma}}e^{U_0}f^2\,dx
&=\sigma\int_0^1F^2\,ds,\\
\int_{B_{R_\sigma}}|\nabla f|^2\,dx
&=4\pi\int_0^1s\left(1-\frac{\sigma s}{8\pi}\right)(F')^2\,ds.
\end{aligned}
\end{equation}
We use the following two inequalities for $G\in C^1([0,1])$ with $G(1)=0$:
\begin{align}
&\int_0^1s(G')^2\,ds
\ge\kappa^2\int_0^1G^2\,ds,\label{ineq1}\\
&\int_0^1s\left(1-\frac{s}{2}\right)(G')^2\,ds
\ge\int_0^1G^2\,ds.\label{ineq2}
\end{align}
For the first inequality \eqref{ineq1}, set $H(t)=G(t^2)$ and use the first
Dirichlet eigenvalue $j_{0,1}^2$ of the unit disk. If $G\not\equiv0$, then
\[
\frac{\int_0^1s(G')^2\,ds}{\int_0^1G^2\,ds}
=\frac14\frac{\int_0^1t(H')^2\,dt}{\int_0^1tH^2\,dt}
\ge\frac{j_{0,1}^2}{4}=\kappa^2.
\]
For the second inequality \eqref{ineq2}, set $a(s)=s(1-s/2)$ and $w(s)=1-s$.
Since $-(aw')'=w$, integration by parts gives
\[
\begin{aligned}
\int_0^1a(G')^2\,ds-\int_0^1G^2\,ds
&=\int_0^1aw^2\left[\left(\frac{G}{w}\right)'\right]^2\,ds
\ge0,
\end{aligned}
\]
where
the boundary terms vanish because $a(0)=0$ and $G(1)=0$.
No boundary condition is imposed at $s=0$.

Since $0<\sigma\le4\pi$ and
\[
s\left(1-\frac{\sigma s}{8\pi}\right)
=\left(1-\frac{\sigma}{4\pi}\right)s
+\frac{\sigma}{4\pi}s\left(1-\frac{s}{2}\right),
\]
applying the two inequalities \eqref{ineq1}-\eqref{ineq2} to $F$ and using \eqref{eq:model-barta}  yields
\begin{align}\label{eq:model-spectral}
\int_{B_{R_\sigma}}|\nabla f|^2\,dx
\ge
\left(\frac{4\pi\kappa^2}{\sigma}+1-\kappa^2\right)
\int_{B_{R_\sigma}}e^{U_0}f^2\,dx.
\end{align}
The same estimate \eqref{eq:model-spectral} holds for every radial $f\in H_0^1(B_{R_\sigma})$
by density.

Therefore, applying \eqref{eq:model-spectral} to \(f=v^*\), and using
\eqref{eq:L2-equimeasurable} and \eqref{eq:PS-weighted}, we conclude that
\[
\begin{aligned}
\int_{\Omega} e^Uv^2\,dx
&=\int_{B_{R_{\sigma}}}e^{U_0}(v^*)^2\,dx\\
&\le
\frac{\sigma}{4\pi\kappa^2+(1-\kappa^2)\sigma}
\int_{B_{R_{\sigma}}}|\nabla v^*|^2\,dx\\
&\le
\frac{\sigma}{4\pi\kappa^2+(1-\kappa^2)\sigma}
\int_{\Omega}|\nabla v|^2\,dx.
\end{aligned}
\]
This proves \eqref{eeffc} for
\(v\in C_c^\infty(\Omega)\) with $v\geq 0$. For \(0\leq v\in H_0^1(\Omega)\), choose
\(0\leq v_k\in C_c^\infty(\Omega)\) such that \(v_k\to v\) in \(H_0^1(\Omega)\).
After passing to a subsequence, \(v_k\to v\) almost everywhere; hence,
by Fatou's lemma,
\[
\int_{\Omega}e^Uv^2\,dx
\le
\liminf_{k\to\infty}\int_{\Omega}e^Uv_k^2\,dx
\le
\frac{\sigma}{4\pi\kappa^2+(1-\kappa^2)\sigma}
\lim_{k\to\infty}\int_{\Omega}|\nabla v_k|^2\,dx,
\]
which proves \eqref{eeffc} for \(0\leq v\in H_0^1(\Omega)\). Finally, by considering $v=\max\{v,0\}-\max\{0,-v\}$, we obtain \eqref{eeffc} for general \(v\in H_0^1(\Omega)\).
\end{proof}

\begin{lemma}\label{Simply nodal set Lem}
	Let $\phi\neq 0$ be a solution of \eqref{eq:jacobi-intro}. If $\rho\leq 4\pi/N$, then $\phi$ has no simply connected nodal domain.
\end{lemma}
\begin{proof}
	
	Let \(\Omega\) be a simply connected nodal domain of \(\phi\). Replacing \(\phi\) by
	\(-\phi\) if necessary, we may assume
	\[
	\Delta\phi+e^u\phi=0, \qquad  \phi>0\quad\text{in }\Omega,
	\qquad
	\phi=0\quad\text{on }\partial \Omega.
	\]
	Testing this equation with \(\phi\), we obtain
	\begin{equation}\label{eq:Jacobi-energy-on-nodal-domain}
		\int_\Omega|\nabla\phi|^2\,dA
		=
		\int_\Omega e^u\phi^2\,dA.
	\end{equation}
	Set
	$
	N\rho_\Omega:=\int_\Omega e^u\,dA,
	$
	since \(\Omega\subsetneqq\mathbb T\),
	\[
	\rho_\Omega<\rho\leq4\pi/N.
	\]
	The Lemma \ref{lem:weighted-spectral-positive-singularities} gives
	\[
	\int_\Omega e^u\phi^2\,dA
	\le
	\frac{N\rho_\Omega}{4\pi}
	\int_\Omega|\nabla\phi|^2\,dA.
	\]
	Together with \eqref{eq:Jacobi-energy-on-nodal-domain}, this yields
	\[
	\int_\Omega|\nabla\phi|^2\,dA
	\le
	\frac{N\rho_\Omega}{4\pi}
	\int_\Omega|\nabla\phi|^2\,dA.
	\]
	Since \(\phi\not\equiv0\), we conclude that
	\[
	\rho_\Omega\ge4\pi/N,
	\]
	contradicting \(\rho_\Omega<4\pi/N\). Therefore no nodal domain of
	\(\phi\) is simply connected.
\end{proof}

The following structural theorem for nodal sets has a long history. The classical version is due to Cheng \cite{ChengNodal}. When \(h\in L^{\infty}(M)\), Helffer et al. \cite{HHOT09} proved that a similar structural result for nodal sets remains valid with \(C^1\) regularity.

\begin{theorem}
	\label{thm:cheng-nodal}
	Let \(M\) be a two-dimensional smooth Riemannian manifold, and let
	\(h\in L^\infty(M)\). Suppose that \(f\not\equiv0\) satisfies
	\[
	(\Delta+h(x))f=0
	\qquad\text{on }M .
	\]
	Then the nodal set
	\[
	Z(f):=\{x\in M:f(x)=0\}
	\]
	has the following structure:
	\begin{enumerate}
		\item The critical zeros on the nodal lines, namely the points
		\[
		\{x\in M:f(x)=0,\ \nabla f(x)=0\},
		\]
		are isolated.
		
		\item If nodal lines meet at a critical zero \(q\), then they form an
		equiangular system. More precisely, if \(q\) is a zero of order \(m\ge2\),
		then \(Z(f)\) consists locally of exactly \(2m\) \(C^1\)-arcs meeting at
		\(q\) with equal angles.
		
		\item The nodal lines consist of a number of \(C^1\)-immersed
		one-dimensional closed submanifolds. In particular, if \(M\) is compact,
		then the nodal lines are a finite union of \(C^1\)-immersed circles.
	\end{enumerate}
\end{theorem}

\begin{lemma}\label{lem:annular-alt}
If $\rho\leq 4\pi/N$, then the set
	$$
	Z:=\{x\in\T:\phi(x)=0\}=\cup_{j=0}^{m-1}\Gamma_j
	$$
	is a finite union of pairwise disjoint essential simple closed curves, all in the
	same free homotopy class, which means that 
	\begin{align*}
		\T=\cup_{j=1}^{m}A_j\qquad\mathrm{and}\qquad \partial A_j=\Gamma_{j-1}\cup \Gamma_{j},
	\end{align*}
	where $A_j$ is a topology cylinder.
\end{lemma}

\begin{proof}
	We first recall the local structure of \(Z\). Recall
	\[
	\Delta \phi+e^u\phi=0
	\qquad\text{in }\T.
	\]
Since $\phi$ satisfies $\int_{\T}e^u\phi=0$, then $\phi$ must change sign, i.e., $Z\not=\emptyset$. Near each singular point $q_j$, we have $e^u=O\bigl(|z-q_j|^{\frac{\beta_j}{2\pi}}\bigr)$, and hence $e^u$ is bounded. Therefore, $e^u\in L^\infty(\T)$.
	 From Theorem~\ref{thm:cheng-nodal},  the nodal lines are
	\(C^1\)-immersed circles. Their critical points are isolated, hence finite on
	the compact torus. Away from these critical points, \(Z\) is a union of smooth
	embedded arcs. If \(q\in Z\) is a critical point, then locally \(2m\) arcs meet
	at \(q\) for some \(m\ge2\). Thus, after declaring all critical points as
	vertices and the smooth arcs between them as edges, \(Z\) becomes a finite
	nodal graph on \(\T\), and every vertex has degree at least \(4\).

	\medskip
	\noindent
	\textbf{Step 1. \(Z\) has no vertices.}
	
	\begin{center}
		\AnnCaseVertex
	\end{center}
	
	Suppose otherwise. Let \(Z_v\) be the union of those connected components of
	\(Z\) which contain vertices, and let \(V\) and \(E\) be the number of vertices
	and edges of \(Z_v\). For \(q\in Z_v\), set
	\[
	\deg(q):=\text{the number of edges incident to }q .
	\]
	Since every vertex has degree at least \(4\), we have
	\[
	2E=\sum_{q\in Z_v}\deg(q)\ge 4V.
	\]
	Hence
	\[
	\chi(Z_v)=V-E\le -V<0.
	\]
	All remaining components of \(Z\), if any, are immersed circle components
	without critical points, hence embedded circles, and have Euler characteristic
	zero. Therefore
	\[
	\chi(Z)<0.
	\]
	
	Let \(U\) be a sufficiently small closed regular neighborhood of the finite graph \(Z\subset T\), and set
\[
S:=T\setminus \operatorname{int}U.
\]
Then \(U\) deformation retracts onto \(Z\). Moreover, the components \(S_D\) of \(S\) are naturally indexed by the nodal domains \(D\) of \(T\setminus Z\), and \(S_D\hookrightarrow D\) is a homotopy equivalence. 
Each \(S_D\) is a compact, connected, orientable surface with nonempty boundary. By Lemma \ref{Simply nodal set Lem}, \(D\) is not simply connected. Since \(S_D\simeq D\), the surface \(S_D\) is not a disk. If \(g_D\) is its genus and \(b_D\ge1\) is the number of its boundary components, then
\[
\chi(S_D)=2-2g_D-b_D\le0.
\]
On the other hand,
\[
\chi(U)=\chi(Z)<0,
\qquad
U\cap S=\partial U.
\]
Since \(U\) is a compact surface with boundary, \(\partial U\) is a compact one-dimensional manifold without boundary. Hence it is a finite disjoint union of circles, and therefore \(\chi(\partial U)=0\). Therefore,
\[
0=\chi(T)
 =\chi(U)+\chi(S)-\chi(\partial U)
 =\chi(Z)+\sum_D\chi(S_D)<0,
\]
a contradiction. Thus \(Z\) has no vertices.

	Therefore \(Z\) is a finite union of pairwise disjoint smooth simple closed
	curves:
	\[
	Z=\gamma_1\cup\cdots\cup\gamma_k .
	\]
	Each component is essential: otherwise, choosing an innermost null-homotopic component would produce a disk whose interior is a simply connected nodal domain, contradicting Lemma \ref{Simply nodal set Lem}.

	\medskip
	\noindent
	\textbf{Step 2. All components have the same free homotopy class.}
	
	\begin{center}
		\AnnCaseParallel
	\end{center}
	
	Fix one component, say \(\gamma_1\). Since \(\gamma_1\) is an essential simple
	closed curve on the torus, its homology class is a nonzero primitive element of
	\[
	H_1(\T;\mathbb Z)\simeq\mathbb Z^2 .
	\]
	After changing the basis of \(H_1(\T;\mathbb Z)\), we may assume
	\[
	[\gamma_1]=(1,0).
	\]
	Let \(\gamma_j\), \(j\ge2\), be another component of \(Z\). Since
	\(\gamma_j\cap\gamma_1=\varnothing\), their algebraic intersection number is
	zero. Write
	\[
	[\gamma_j]=(m,n)\in H_1(\T;\mathbb Z).
	\]
	Then
	\[
	[\gamma_1]\cdot[\gamma_j]
	=
	\det
	\begin{pmatrix}
		1 & 0\\
		m & n
	\end{pmatrix}
	=
	n .
	\]
	Hence \(n=0\), and therefore
	\[
	[\gamma_j]=(m,0)=m[\gamma_1].
	\]
	Since \(\gamma_j\) is also an essential simple closed curve, its homology class
	is nonzero and primitive. Thus \(m=\pm1\). Consequently,
	\[
	[\gamma_j]=\pm[\gamma_1].
	\]
	Therefore \(\gamma_j\) and \(\gamma_1\) have the same unoriented free homotopy
	class. Hence all components of \(Z\) are parallel essential simple closed
	curves.
	
	Finally, a finite collection of pairwise disjoint parallel essential simple
	closed curves cuts the torus into a disjoint union of annuli. Therefore every
	nodal domain of \(\phi\) is an annulus. Since \(\phi\) changes sign, \(\T\setminus Z\)
	cannot consist of a single nodal domain. In particular, there are at least two
	annular nodal domains.
\end{proof}

Let \(A\) be one such annular nodal domain.  Replacing \(\phi\) by \(-\phi\)
if necessary, we may assume
\[
\phi>0\quad\text{in }A,\qquad \phi=0\quad\text{on }\partial A .
\]
On \(A\), the restrictions of \(u\) and \(\phi\) satisfy
\[
\Delta\phi+e^u\phi=0
\qquad\text{in }A,
\]
and, in the sense of distributions,
\[
\Delta u+e^u
=
\sum_{q\in\{p_1,\cdots,p_{N}\}\cap A}\beta_q\delta_q
\ge0 .
\]
If one of \(p_1,\cdots,p_N\) lies on \(\partial A\), the same formulas below are
understood by first removing a small disk around that boundary singular point
and then letting its radius tend to zero.  Since the singularities are positive,
this approximation does not decrease the relevant lower bounds.

Set
\[
C_\ell=(0,\ell)\times \mathbb{S}^1=\{(t,\theta): 0\leq t\leq l, 0\leq\theta\leq 2\pi\},\qquad |\mathbb{S}^1|=2\pi,
\]
then
choose a conformal map
\[
F:C_\ell:=(0,\ell)\times \mathbb{S}^1\longrightarrow A,
\qquad \mathbb{S}^1=\mathbb R/(2\pi\mathbb Z),
\]
where \(\ell>0\) is the conformal length of the annulus.  Write
\[
F^*g_{\T}=e^{2\omega}(dt^2+d\theta^2).
\]
Equivalently, in a local complex coordinate, \(e^\omega=|F'|\), and
\(\omega\) is harmonic.  Define the functions on \(C_\ell\)
\[
U:=u\circ F+2\omega,
\qquad
\psi:=\phi\circ F .
\]
This is the important point: the cylindrical Liouville potential is not
\(u\circ F\), but \(u\circ F+2\omega\).

We now compute the transformed equations.  For any smooth function \(f\) on
\(A\),
\[
\Delta_{C_\ell}(f\circ F)
=
e^{2\omega}(\Delta_A f)\circ F .
\]
Since \(\Delta_{C_\ell}\omega=0\), we obtain, in the sense of distributions,
\[
\begin{aligned}
	\Delta_{C_\ell} U+e^U
	&=
	\Delta_{C_\ell}(u\circ F+2\omega)
	+e^{u\circ F+2\omega}  \\
	&=
	e^{2\omega}(\Delta_Au)\circ F
	+e^{2\omega}(e^u\circ F)  \\
	&=
	e^{2\omega}\bigl(\Delta_Au+e^u\bigr)\circ F .
\end{aligned}
\]
Hence
\[
\Delta_{C_\ell} U+e^U
=
\sum_{q\in\{p_1,\cdots,p_{N}\}\cap A}\beta_q\delta_{F^{-1}(q)}
\ge0
\qquad\text{on }~ \overline{C_\ell} .
\]

Similarly, since \(\psi=\phi\circ F\), we have
\[
\begin{aligned}
	\Delta_{C_\ell}\psi+e^U\psi
	&=
	e^{2\omega}(\Delta_A\phi)\circ F
	+
	e^{u\circ F+2\omega}(\phi\circ F) \\
	&=
	e^{2\omega}\bigl(\Delta_A\phi+e^u\phi\bigr)\circ F
	=
	0.
\end{aligned}
\]
Moreover,
\[
\psi>0\quad\text{in }C_\ell,
\qquad
\psi=0\quad\text{on }\partial C_\ell.
\]

The mass is also preserved under this transformation.  Indeed,
\[
e^U\,dt\,d\theta
=
e^{u\circ F}e^{2\omega}\,dt\,d\theta
=
F^*(e^u\,dA_T).
\]
Therefore
\[
\int_{C_\ell}e^U\,dt\,d\theta
=
\int_A e^u\,dA_T.
\]
For an annulus \(A\) with boundary \(\partial A= \partial A_1\cup \partial A_2\), its capacity is
\begin{equation}\label{eq:capdef}
	\begin{aligned}
		\Cap(A):=&\inf\Bigl\{
		\int_A |\nabla \eta|^2\,\ud A:\ 
		\eta\in H^1(A),
		\eta=0 \,\text{ on}\, \partial A_1,  
		\eta=1\,\text{ on}\, \partial A_2\Bigr\}\\
        =&\inf\Bigl\{
		\int_{C_\ell} |\nabla \eta|^2\,\ud t\ud \theta:\ 
		\eta\in H^1(C_\ell),
		\eta(0,\theta)=0, \eta(l,\theta) =1\Bigr\}.
	\end{aligned}
\end{equation}
The capacity of a doubly connected domain can be defined variationally through the Dirichlet energy; see, for instance, \cite{BerlyandMironescu2003,Ziemer1967} and the references therein.
This quantity is conformally invariant in two dimensions, and for the flat cylinder \(C_\ell=(0,\ell)\times \mathbb S^1\), from \cite{BerlyandMironescu2003,NasserVuorinen2021}, one has
\[
\operatorname{Cap}(C_\ell)=\frac{2\pi}{\ell}.
\]

\section{The Alexandrov–Bol type inequality on cylinder}\label{Sec 4}
The following inequality can be viewed as a Alexandrov–Bol type inequality on the cylinder. It cannot be derived from the simply connected-domain case in \(\C\), since the first eigenfunction \(\psi\) on the cylinder cannot be regarded as the first eigenfunction of a simply connected domain in \(\C\); see \cite{GuiMoradifam,LinLucia07}.

\begin{theorem}\label{thm:local}
	Let \(C_\ell=(0,\ell)\times \mathbb{S}^1\), \(|\mathbb{S}^1|=2\pi\). Assume
	\begin{equation}\label{eq:Usub}
		\Delta U+e^U= \sum_{q\in\{p_1,\cdots,p_{N}\}\cap A}\beta_q\delta_{F^{-1}(q)},
		\quad\text{ on } ~\overline{C_\ell},
	\end{equation}
where $\beta_q>0$ for all $q\in\{p_1,\cdots,p_{N}\}$	and assume that there is a function \(\psi\) such that
	\begin{equation}\label{eq:jaclocal}
		\Delta\psi+e^U\psi=0,
		\qquad
		\psi>0~\text{ in }~C_\ell,
		\qquad
		\psi=0~\text{ on }~\partial C_\ell.
	\end{equation}
	Then, set \(c=\Cap(C_\ell)=2\pi/\ell\), we have
	\begin{equation}\label{eq:phisimple}
		\int_{C_\ell}e^U> \frac{4\pi\kappa c}{2+\kappa c}.
	\end{equation}
\end{theorem}
\begin{proof}
	Fix \(n\in\mathbb N\). Lift \(U\) and \(\psi\) to the cyclic \(n\)-cover
	\[
	C_\ell^{(n)}=(0,\ell)\times\bigl(\R/(2\pi n\Z)\bigr),
	\]
	and denote the lifts by \(U_n\) and \(\psi_n\). Put
	\[
	M:=\int_{C_\ell}e^U,
	\qquad
	E_n:=\int_{C_\ell^{(n)}}|\nabla\psi_n|^2,
	\qquad
	P_n:=\int_{C_\ell^{(n)}}\psi_n^2.
	\]
	Since \(\psi_n\) solves the Jacobi equation,
	\begin{equation}\label{eq:energyidentity}
		E_n=\int_{C_\ell^{(n)}}e^{U_n}\psi_n^2.
	\end{equation}
	
	Define the Lipschitz partition
	\[
	\chi_{1,n}(\theta)=\left|\cos\frac{\theta}{2n}\right|,
	\qquad
	\chi_{2,n}(\theta)=\left|\sin\frac{\theta}{2n}\right|.
	\]
	Then
	\begin{equation}\label{eq:partition}
		\chi_{1,n}^2+\chi_{2,n}^2=1,
		\qquad
		|\nabla\chi_{1,n}|^2+|\nabla\chi_{2,n}|^2=\frac1{4n^2}
		\quad\text{a.e.}\quad C_\ell^{(n)}.
	\end{equation}
	Without loss of generality, we can parameterize $ C_\ell^{(n)}$ as 
	\begin{align*}
		C_\ell^{(n)}=\{(t,\theta)\in \C: t\in (0,\ell) \quad\mathrm{and}\quad \theta\in (0, 2\pi n)\}
	\end{align*}
	For \(\chi_{1,n}\psi_n\), we cut the cylinder \(C_\ell^{(n)}\) along \(\theta=\pi n\); whereas for \(\chi_{2,n}\psi_n\), we cut the cylinder \(C_\ell^{(n)}\) along \(\theta=0\). Thus they may be regarded as functions in \(H_0^1\) on the corresponding simply connected cut cylinders.
	
	If \(nM\le4\pi\), Lemma \ref{lem:weighted-spectral-positive-singularities} gives
	\[
\left(\frac{4\pi\kappa^2}{nM}+1-\kappa^2\right)	\int_{C_\ell^{(n)} }  e^{U_n}\chi_{i,n}^2\psi_n^2
	\le
\int_{C_\ell^{(n)} } |\nabla(\chi_{i,n}\psi_n)|^2.
	\]
	Summing over \(i=1,2\), using \eqref{eq:partition} and \begin{align*}
		\chi_{1,n}\nabla \chi_{1,n}+\chi_{2,n}\nabla\chi_{2,n}=0,
	\end{align*}
	gives
	\begin{equation}\label{eq:imsbound}
	\left(\frac{4\pi\kappa^2}{nM}+1-\kappa^2\right)	E_n\le 
		E_n+\frac1{4n^2}P_n.
	\end{equation}
	Since \(\psi_n=0\) at \(t=0,\ell\), the standard  Poincare inequality  gives
	\begin{equation}\label{eq:poincare}
		P_n\le \frac{\ell^2}{\pi^2}E_n.
	\end{equation}
	Combining \eqref{eq:imsbound} and \eqref{eq:poincare}, and dividing by \(E_n>0\), we obtain
	\[
\frac{4\pi\kappa^2}{nM}+1-\kappa^2	\le 1+\frac{\ell^2}{4\pi^2n^2}.
	\]
	Notice that this inequality also holds if $nM>4\pi$. Hence
 \begin{equation}\label{eq:MNbound}
		M\ge 4\pi\sup_{n\in \mathbb{N}}\frac{ n}{n^2+\ell^2/(4\pi^2\kappa^2)}=4\pi \sup_{n\in \mathbb{N}}\frac{ n\kappa^2c^2 }{n^2{\kappa^2}c^2+1},
	\end{equation}
	where \(c=2\pi/\ell\) is used. Finally, if $\kappa  c\leq 1$, choose
	\(n=\lceil 1/(\kappa  c)\rceil\). Then \(1\le n \kappa  c<1+\kappa  c\), and since
	\(x/(1+x^2)\) is decreasing on \([1,\infty)\),
	\[
	\frac{n\kappa^2c^2}{1+n^2\kappa^2c^2}
	=\kappa c\frac{n\kappa c}{1+(n\kappa c)^2}
	\ge \kappa c\frac{1+\kappa c}{1+(1+\kappa c)^2}
	> \frac{\kappa c}{2+\kappa c}.
	\]
	If $\kappa c>1$, then 
	\begin{align*}
		\sup_{n\geq 1}\frac{n\kappa^2 c^2}{1+n^2\kappa^2c^2}
	\geq 	\frac{\kappa^2c^2}{1+\kappa^2c^2}>\frac{\kappa c}{2+\kappa c}.
	\end{align*}This proves \eqref{eq:phisimple}.
\end{proof}

In general, the lower bound \eqref{eq:phisimple} is not sharp. However, if the solution is one-dimensional, we can obtain a sharp version.

\begin{lemma}
\label{lem:one-dimensional-mass}
Let $v\in W^{1,\infty}_{\mathrm{loc}}(0,L)$ satisfy
$e^v\in L^1(0,L)$ and, in the distributional sense,
\[
 v''+e^v=\mu,
 \qquad \mu\geq0,
\]
where $\mu$ is a nonnegative  Radon measure on $(0,L)$,
not necessarily of finite total mass. Suppose that there exists
$p\in H^1_0(0,L)$ satisfying $p''+e^vp=0$ and
$p>0$ in $(0,L)$. Then
\[
 L\int_0^L e^v\,dx\geq8.
\]
If $\mu$ is finite, equality holds if and only if $\mu=0$ and
\begin{align}\label{eq:One dim case}
 e^{v(x)}
 =2a^2\operatorname{sech}^2\!\left(a(x-L/2)\right),
 \qquad a=\frac{2T}{L},
\end{align}
where $T>0$ is the unique solution of
$T\tanh T=1$.
\end{lemma}

\begin{proof}
We first assume that $\mu((0,L))<\infty$.
Since $v''=\mu-e^v\,dx$ is a finite signed measure,
$v'$ has bounded variation and $v$ extends to a
Lipschitz function on $[0,L]$.
Set $Q=\int_0^L e^v\,dx$ and introduce
\[
 t=t(x)=\frac1Q\int_0^x e^{v(s)}\,ds,
 \qquad
 r(t)=\frac{e^{v(x(t))}}{Q^2},
 \qquad 0\leq t\leq1.
\]
We use the same notation for the transformed Jacobi field.
Since $r_t=v_x/Q$ and $r_{tt}=v_{xx}/e^v$, the equations and the quantity
to be estimated become
\[
 r''+1=\nu_0,
 \qquad
 -(rp')'=p,
 \qquad
 \int_0^1\frac{dt}{r(t)}=LQ,
\]
where $'=d/dt$, $\nu_0:=Q^{-1}t_{\#}\mu$ is the normalized
push-forward of $\mu$ under $x\mapsto t(x)$ and $p(t)=p(x(t))$.
Thus $\nu_0$ is a finite nonnegative measure on $(0,1)$,
and $p(0)=p(1)=0$.
In particular, $r$ is positive on $[0,1]$, and the
positivity of $p$ implies $\lambda_1(r)=1$ (since only the first eigenfunction does not change sign).
Here $\lambda_1(r)$ denotes the first Dirichlet
eigenvalue of $-(r\,\cdot\,')'$, that is,
\begin{align}\label{eq:eigenvalue}
 \lambda_1(r)
 =\inf_{0\neq f\in H_0^1(0,1)}
 \frac{\int_0^1r(f')^2}{\int_0^1f^2}.
\end{align}
Thus it suffices to prove that
\begin{align}\label{Rr}
 R(r):=\int_0^1\frac{\ud t}{r(t)}\geq8.
\end{align}

\textbf{Step 1: the minimizing problem.}
Fix $\nu_0$ and consider
\begin{align}\label{eq minimizing}
 \min_{r\in\mathcal M}R(r),
\end{align}
where
\begin{align*}
 \mathcal M:=&\left\{
 r:r(t)=a+bt-\frac{t^2}{2}
 +\int_{0}^{1}(t-s)_+\,\ud\nu(s),
 \quad a,b\in\mathbb R,\right.\\
 &\qquad\qquad\left.
 0\leq\nu\leq\nu_0,\quad
 r(t)>0\text{ for }t\in[0,1],\quad
 \lambda_1(r)=1
 \right\}.
\end{align*}
Clearly, $\mathcal M\neq\emptyset$, since it contains
the standard model \eqref{eq:One dim case}, with $\nu=0$.
Choose a minimizing sequence with $R(r_n)\leq C_0$.
We first justify that the minimum is attained.
It is enough to establish, on this sublevel set,
\begin{align}\label{eq uniform bound}
 |a|,\ |b|,\ \|r\|_{C^{0,1}([0,1])}\leq C
 \qquad\mathrm{and}\qquad r(t)\geq\delta_0>0,
\end{align}
where the constants may depend on $C_0$ and $\nu_0$.

We use the elementary inequalities
\begin{align}\label{eq:Funda ineq}
 \int_0^1f^2\leq\int_0^1t|f'|^2,
 \qquad
 \int_0^1f^2\leq\int_0^1(1-t)|f'|^2,
 \qquad f\in H_0^1(0,1).
\end{align}
Since $a=r(0)>0$, we have $r(t)\geq(b-\tfrac12)t$. Next, we claim that $b\leq 3/2$.
If $b>1/2$, the \eqref{eq:eigenvalue} and the first inequality in \eqref{eq:Funda ineq} give
\begin{align*}
 1=\lambda_1(r)
 \geq\left(b-\frac12\right)
 \inf_{0\neq f\in H_0^1(0,1)}
 \frac{\int_0^1t(f')^2}{\int_0^1f^2}
 \geq b-\frac12.
\end{align*}
Thus $b\leq3/2$, which is immediate also when $b\leq1/2$.
Moreover, $r'(t)=b-t+\nu((0,t))$ almost everywhere,
so $r'(t)\leq b+\nu_0((0,1))$ and
\begin{align*}
 r(t)=r(1)-\int_t^1r'(s)\,\ud s
 \geq\bigl(-b-\nu_0((0,1))\bigr)(1-t).
\end{align*}
If $b+\nu_0((0,1))<0$, multiplying $|p'|^2$ on two sides, we can see
\begin{align*}
    \int_0^1p^2= \int_0^1r|p'|^2\geq -(b+\nu_0((0,1)))\int_0^1(1-t)|p'|^2,
\end{align*}
 the second inequality in \eqref{eq:Funda ineq}
gives 
$1\geq-b-\nu_0((0,1))$; otherwise if $b+\nu_0(0,1)\geq 0$, then $b\geq-\nu_0(0,1) $. Consequently,
\[
 -1-\nu_0((0,1))\leq b\leq\frac32.
\]
Hence $r'$ is uniformly bounded.
Moreover, the ordinary Poincar\'e inequality gives
\begin{align*}
 1=\lambda_1(r)
 \geq\min_{t\in[0,1]}r(t)
 \inf_{0\neq f\in H_0^1(0,1)}
 \frac{\int_0^1(f')^2}{\int_0^1f^2}
 =\pi^2\min_{t\in[0,1]}r(t).
\end{align*}
It follows that
\begin{align*}
 a=r(0)
 \leq\min_{t\in[0,1]}r(t)+\|r'\|_{L^\infty(0,1)}
 \leq C.
\end{align*}

Finally, the bound $R(r)\leq C_0$ keeps $r$ uniformly
away from zero. Otherwise, after passing to a subsequence,
$m_n:=r_n(t_n)=\min_{[0,1]}r_n\to0$.
If $t_n\geq1/2$, the uniform Lipschitz bound gives
\begin{align*}
 \int_0^1\frac{\ud t}{r_n(t)}
 \geq\int_0^{t_n}\frac{\ud t}{m_n+C(t_n-t)}
 =\frac1C\log\frac{m_n+Ct_n}{m_n}
 \longrightarrow+\infty.
\end{align*}
If $t_n<1/2$, the same argument applies on $[t_n,1]$.
Both cases contradict $R(r_n)\leq C_0$.
This proves \eqref{eq uniform bound}.

Extend the measures by zero to $[0,1]$.
Since $0\leq\nu_n\leq\nu_0$, after passing to a subsequence,
$\nu_n\rightharpoonup^*\nu_*$ with $0\leq\nu_*\leq\nu_0$,
and the affine coefficients also converge.
The integral representation and the uniform Lipschitz bound
then give $r_n\to r_*$ uniformly on $[0,1]$, with
$r_*\geq\delta_0$.
The Rayleigh characterization yields
$\lambda_1(r_n)\to\lambda_1(r)$, while
$R(r_n)\to R(r_*)$.
Thus $r_*\in\mathcal M$ attains \eqref{eq minimizing}.

\medskip
\textbf{Step 2: Rule out the positive measure.}
We claim that $\nu_*=0$ at a minimizer
\begin{align*}
    r_*=a_*+b_*t-\frac{t^2}{2}
 +\int_{0}^{1}(t-s)_+\,\ud\nu_*(s).
\end{align*}
To see this, we first record a convenient formula for
variations preserving the constraint $\lambda_1(r_*)=1$.
Let $p>0$ be the first eigenfunction of
$$
 -(r_*p')'=p,\qquad p(0)=p(1)=0,
$$
normalized by $\int_0^1p^2 dt=1$.
For a fixed continuous variation $h$, there exists $\alpha_n(\varepsilon)$ such that
$$
 r_\varepsilon=r_*+\varepsilon h+\alpha_h(\varepsilon),
 \qquad \alpha_h(0)=0,
$$
and $\lambda_1(r_\varepsilon)=1$. This follows by considering the function
\begin{align*}
G(\varepsilon,\alpha)=\lambda_1(r_*+\varepsilon h+\alpha)-1
\end{align*}
Since $G(0,0)=0$ and \begin{align*}
\partial_{\alpha}G(0,0)=D\lambda_1(r_*)[1]=\int_0^1(p')^2\,dt>0,
\end{align*}
such a function $\alpha_h(\varepsilon)$ exists by the implicit
function theorem. Here, we note that for an arbitrary $h$, $r_{\varepsilon}$ need not belong to $\mathcal{M}$. Differentiating the constraint gives

$$
 \alpha_h'(0)
 =-\frac{\displaystyle\int_0^1(p')^2h\,dt}
         {\displaystyle\int_0^1(p')^2\,dt}.
$$

Consequently,
\begin{align}\label{eq: variation formula-1}
\left.\frac{d}{d\varepsilon}R(r_\varepsilon)
\right|_{\varepsilon=0}
=\int_0^1F(t)h(t) dt,
\end{align}
where
\begin{align}\label{eq:variation formula-2}
\eta:=
\frac{\displaystyle\int_0^1r_*^{-2}dt}
{\displaystyle\int_0^1(p')^2dt},
\qquad F:=\eta(p')^2-r_*^{-2},\qquad \int_0^1F dt=0
\end{align}

Next, we choose two special admissible perturbations to prove that $\nu_*=0$ for the minimizer $r_*$.
\begin{itemize}
\item The first admissible perturbation: Taking
$h(t)=t$, we have
\begin{align*}
r_{\varepsilon}=(a_*+\alpha_h(\varepsilon))+(b_*+\varepsilon)t-\frac{t^2}{2}
+\int_{0}^{1}(t-s)_+\ud\nu_*(s),
\end{align*}
Then, for $\varepsilon$ sufficiently small, by \eqref{eq uniform bound}, we obtain $r_{\varepsilon}\in \mathcal{M}$. Moreover, from \eqref{eq: variation formula-1}, we have
\begin{align}\label{variation formula-3}
\int_0^1tF(t)dt=0.
\end{align}
\item  The second admissible perturbation: Taking
$
h(t)=-\int_{0}^{1}(t-s)_+\ud\nu_*(s),
$
we have
\begin{align*}
r_{\varepsilon}=(a_*+\alpha_h(\varepsilon))+b_*t-\frac{t^2}{2}
+(1-\varepsilon)\int_{0}^{1}(t-s)_+\ud\nu_*(s),
\end{align*}
For small $\varepsilon>0$, this replaces $\nu_*$ by
$(1-\varepsilon)\nu_*$. Using \eqref{eq uniform bound}, we obtain $r_{\varepsilon}\in \mathcal{M}$. Similarly, we also have
\begin{align}\label{eq K}
\left.\frac{d}{d\varepsilon}R(r_\varepsilon)
\right|_{\varepsilon=0^+}
=-\int_0^1K(s)\ud\nu_*(s)\geq0,
\quad
K(s):=\int_s^1(t-s)F(t) dt.
\end{align}
\end{itemize}
However, we claim that \eqref{eq:variation formula-2} and \eqref{variation formula-3} force
$K(s)>0$ for every $s\in(0,1)$,
which contradicts \eqref{eq K} if $\nu_*\not=0$.

Put $g=r_*p'$. Since $g'=-p<0$, the function $g$ is
strictly decreasing, with $g(0)>0$ and $g(1)<0$.
Since
$$
 F=\frac{\eta g^2-1}{r_*^2},
$$
the identities
\begin{equation}\label{eq:variation formula-20}
 \int_0^1F\,dt=0,
 \qquad
 \int_0^1tF\,dt=0
\end{equation}
imply that $F$ must change sign. In fact, $F$ has at most one negative region, so only three cases can occur.

\textbf{Case 1:} There exists $s_0\in (0,1)$ such that
$F>0$ on $(0,s_0)$ and $F<0$ on $(s_0,1)$. Using \eqref{eq:variation formula-20}, one has
$$
 0=\int_0^1tF(t)\,dt
 =\int_0^1(t-s_0)F(t)\,dt<0,
$$
a contradiction.

\textbf{Case 2:} There exists $s_0\in (0,1)$ such that
$F<0$ on $(0,s_0)$ and $F>0$ on $(s_0,1)$. Then
$$
 0=\int_0^1tF(t)\,dt
 =\int_0^1(t-s_0)F(t)\,dt>0,
$$
a contradiction.

\textbf{Case 3:} There exists $0<s_1<s_2<1$ such that
$F>0$ on $(0,s_1)\cup (s_2,1)$ and $F<0$ on $(s_1,s_2) $ .
Equation \eqref{eq:variation formula-20} also gives
\begin{align}\label{eq K2}
K(s)=\int_0^s(s-t)F(t) dt=\int_{s}^1(t-s)F(t) dt.
\end{align}
Thus $K>0$ on $(0,s_1)$ by the first identity in
\eqref{eq K2}, and $K>0$ on $(s_2,1)$ by the second identity
in \eqref{eq K2}. Finally, on $(s_1,s_2)$,
$K''=F<0$, so $K$ remains positive by concavity.
Therefore
\[
 K(s)>0\qquad(0<s<1).
\] Thus $\nu_*=0$.

\medskip
\textbf{Step 3: the conclusion for finite measures.}
Write $r_*(t)=a_*+b_*t-t^2/2$ and put
$H(t)=\int_0^t r_*^{-1}\,ds$, so that $H(1)=R(r_*)$.
A direct differentiation shows that $r_*'$ and $2+Hr_*'$
are solutions of $-(r_*f')'=f$. Their Wronskian is
$(2a_*+b_*^2)/r_*>0$, so the solution space is
\begin{align*}
 \mathrm{span}\{r_*',2+Hr_*'\}.
\end{align*}
Thus there exist constants $A,B$, not both zero, such that
\begin{align*}
 p=Ar_*'+B(2+Hr_*').
\end{align*}
Since $r_*'(0)=b_*$, $r_*'(1)=b_*-1$, and $p(0)=p(1)=0$,
we have
\[
 0=\det\begin{pmatrix}
 b_* & 2\\
 b_*-1 & 2+(b_*-1)R(r_*)
 \end{pmatrix}
 =2-b_*(1-b_*)R(r_*).
\]
As $R(r_*)>0$, necessarily $0<b_*<1$. Therefore
\[
 R(r_*)=\frac{2}{b_*(1-b_*)}\geq8,
\]
with equality precisely when $b_*=1/2$.

If equality holds for $r$, then
\begin{align*}
    8=R(r)\geq \min_{r\in\mathcal M}R(r)=R(r_*)=8,
\end{align*}
then $r$ must be a minimizer of $\min_{r\in\mathcal M}R(r)$, 
Step~2 gives $\nu_0=0$,
and hence $\mu=0$, while Step~3 gives $b=1/2$ and so $r(t)=a+t/2-t^2/2$.
Since $r_t=v_x/Q$, we obtain $v'(0)=-v'(L)$.
The source-free equation consequently gives
$e^{v(x)}=2a^2\operatorname{sech}^2(a(x-L/2))$
for some $a>0$.
Its mass is $Q=4a\tanh(aL/2)$, so $LQ=8$ is
equivalent to $s\tanh s=1$, where $s=aL/2$.
Thus $a=2T/L$.
Conversely, this profile, together with
$p(x)=1-a(x-L/2)\tanh(a(x-L/2))$, satisfies the
hypotheses and attains equality.

\medskip
\textbf{Step 4: locally finite measures.}
\medskip
We now allow $\mu((0,L))=+\infty$. Fix $0<\delta<L/2$ and set
$I_\delta=(\delta,L-\delta)$. Since $\mu$ is a Radon measure, we have
$\mu(I_\delta)<\infty$. Define
\[
\lambda_\delta
:=
\inf_{0\neq f\in H_0^1(I_\delta)}
\frac{\displaystyle\int_{I_\delta}|f'|^2\,dx}
{\displaystyle\int_{I_\delta}e^vf^2\,dx}.
\]

We first claim that $\lambda_\delta\geq1$ and
$\lambda_\delta\to1$ as $\delta\downarrow0$. Indeed, since $p>0$
satisfies $p''+e^vp=0$ in $(0,L)$, the ground-state identity gives
\[
\int_{I_\delta}\left(|f'|^2-e^vf^2\right)\,dx
=
\int_{I_\delta}p^2
\left|\left(\frac{f}{p}\right)'\right|^2\,dx
\geq0
\]
for every $f\in C_c^\infty(I_\delta)$. By density, the same inequality
holds for every $f\in H_0^1(I_\delta)$, and hence
$\lambda_\delta\geq1$.

To prove the convergence, choose $p_k\in C_c^\infty(0,L)$ such that
$p_k\to p$ in $H_0^1(0,L)$. Since
$H_0^1(0,L)\hookrightarrow C^0([0,L])$, we also have
$p_k\to p$ uniformly. Moreover,
$\int_0^L|p'|^2\,dx=\int_0^Le^vp^2\,dx$, and therefore
\[
\frac{\displaystyle\int_0^L|p_k'|^2\,dx}
{\displaystyle\int_0^Le^vp_k^2\,dx}
\longrightarrow1.
\]
For each fixed $k$, $\operatorname{supp}p_k\subset I_\delta$ for
$\delta>0$ sufficiently small. Thus
\[
1\leq\lambda_\delta
\leq
\frac{\displaystyle\int_0^L|p_k'|^2\,dx}
{\displaystyle\int_0^Le^vp_k^2\,dx},
\]
which yields $\lambda_\delta\to1$. Now set $v_\delta=v+\log\lambda_\delta$ in $I_\delta$. Then
\[
v_\delta''+e^{v_\delta}
=
\mu|_{I_\delta}
+
(\lambda_\delta-1)e^v
=: \mu_\delta\geq0.
\]
Since $\mu(I_\delta)<\infty$ and $e^v\in L^1(0,L)$, the measure
$\mu_\delta$ is finite. By the definition of $\lambda_\delta$, there
exists a positive first eigenfunction $p_\delta\in H_0^1(I_\delta)$
satisfying
\[
p_\delta''+e^{v_\delta}p_\delta=p_\delta''+\lambda_{\delta}e^{v}p_\delta=0
\qquad\text{in }I_\delta.
\]
Hence the finite-measure case applies to $(v_\delta,p_\delta)$ and gives
\[
(L-2\delta)\lambda_\delta
\int_\delta^{L-\delta}e^v\,dx\geq8.
\]
Letting $\delta\downarrow0$, using $\lambda_\delta\to1$ and
$\int_\delta^{L-\delta}e^v\,dx\to\int_0^Le^v\,dx$, we conclude that
\[
L\int_0^Le^v\,dx\geq8.
\]
\end{proof}

Based on the sharp one-dimensional result, we can use a blow-up analysis to improve the lower bound \eqref{eq:phisimple} to an asymptotically sharp one as $c\to 0$. Given $c=\Cap(C_\ell)>0$, we define the best constant
\begin{align}\label{Def B}
B(c)=\inf_{U\in \mathcal{X}_c,~l=\frac{2\pi}{c}}\int_{C_\ell  }e^U,
\end{align}
where

$$
\mathcal{X}_c
=
\left\{
U:
\begin{aligned}
& \Delta U+e^U=\sum_{j=1}^{m}\beta_{q_j}\delta_{q_j}
 \quad\text{in } \overline{C_\ell},\quad \beta_j\geq 0,l=\frac{2\pi}{c}\\
& \Delta\psi+e^U\psi=0,
\quad \psi>0 \text{ in } C_\ell,
\quad \psi=0 \text{ on } \partial C_\ell
\end{aligned}
\right\}.
$$

	\begin{theorem}\label{thm:small-capacity-asymptotic}
			Let $B(c)$ be defined by \eqref{Def B}. Then 
			\[
			B(c)=8c+o(c)\qquad\hbox{as }c\downarrow0.
			\]
		\end{theorem}
		
		\begin{proof}
			Let $T>0$ satisfy $T\tanh T=1$, and put
			$a=2T/\ell$ and $x=a(t-\ell/2)$.
			The source-free pair
			\[
			U(t)=\log(2a^2)-2\log\cosh x,
			\qquad
			\psi(t)=1-x\tanh x
			\]
			is admissible and satisfies
			\[
			\int_{C_\ell}e^U
			=8\pi a\tanh T
			=\frac{16\pi}{\ell}
			=8c.
			\]
			Thus $B(c)\leq8c$.
			
			To prove the reverse asymptotic, suppose that there exist
			$\delta>0$, $\ell_k\to\infty$, and admissible pairs
			$(U_k,\psi_k)$ such that
			\begin{equation}\label{eq:mass-deficit}
				\ell_k\int_{C_{\ell_k}}e^{U_k}\leq16\pi-\delta.
			\end{equation}
			For each $k$, after omitting the zero coefficients and
			relabeling the singular points if necessary, write
			\[
			\Delta U_k+e^{U_k}
			=
			\sum_{j=1}^{m_k}\beta_{j,k}\delta_{q_{j,k}},
			\qquad
			\beta_{j,k}>0,\qquad	\Delta \psi_k+e^{U_k}\psi_k=0
			\]
			where
			$q_{j,k}=(t_{j,k},\theta_{j,k})\in C_{\ell_k}$.
			In the following, we will make a blow up analysis and cause a 
			contradiction with  Lemma~\ref{lem:one-dimensional-mass}.
			
		{\bf Step 1: Rescaling and local estimate.}
			On $\Omega=(0,1)\times\mathbb S^1$, set
			\[
			\widetilde U_k(s,\theta)
			=U_k(\ell_ks,\theta)+2\log\ell_k,
			\qquad
			W_k=e^{\widetilde U_k},
			\]
			and write
			$\nabla_k=(\partial_s,\ell_k\partial_\theta)$,
			$\Delta_k=\partial_s^2+\ell_k^2\partial_\theta^2$.
			Let the  rescaled Jacobi fields $\widetilde{\psi}_k(s,\theta)=\psi_k(\ell_ks,\theta)$,
			normalized by $\|\widetilde{\psi}_k\|_{L^2(\Omega)}=1$.
			Then
			\begin{equation}\label{eq:wkin}
			    			\int_\Omega W_k
			=\ell_k\int_{C_{\ell_k}}e^{U_k}\leq16\pi-\delta,
			\end{equation}
			and
			\[
			\Delta_k\widetilde{\psi}_k+W_k\widetilde{\psi}_k=0
			\quad\text{in }\Omega,
			\qquad
			\widetilde{\psi}_k>0,\quad \widetilde{\psi}_k\in H^1_0(\Omega).
			\]
			Put $E_k(f)=\int_\Omega|\nabla_kf|^2$.
			Since $W_k>0$, only the first eigenfunction of $-\Delta_k-W_k$ does not change sign. Hence, $\lambda_1(-\Delta_k-W_k)=0$, which means that
			\begin{equation}\label{eq:rescaled-stability}
				\int_\Omega W_kf^2\leq E_k(f)
				\quad(f\in H^1_0(\Omega)),
				\qquad
				E_k(\widetilde{\psi}_k)=\int_\Omega W_k\widetilde{\psi}_k^2.
			\end{equation}
			The rescaled Liouville equation is
			\[
			\Delta_k\widetilde U_k+W_k
			=
			\ell_k\sum_{j=1}^{m_k}
			\beta_{j,k}\delta_{\widetilde q_{j,k}},
			\qquad
			\beta_{j,k}>0,
			\]
			where
			\[
			\widetilde q_{j,k}
			=
			\left(\frac{t_{j,k}}{\ell_k},\theta_{j,k}\right).
			\]
			Take $\eta=\eta(s)\in C_c^\infty(0,1)$, then we claim that $\eta e^{\widetilde U_k/2}\in H^1_0(\Omega)$. In fact, we only need to check the integrability near $\widetilde q_{j,k}$. 
For each fixed $k$, the usual logarithmic expansion near
			$q_{j,k}$ and the rescaling give
			\[
			\widetilde U_k
			=
			\frac{\beta_{j,k}}{2\pi}\log r+O_{j,k}(1),
			\]
			and hence
			\begin{align}\label{eq: asy est}
			W_k=O_{j,k}
			\left(r^{\beta_{j,k}/(2\pi)}\right),
			\qquad
			|\nabla\widetilde U_k|=O_{j,k}(r^{-1}).
			\end{align}
            Hence, near each singularity $\widetilde q_{j,k}$, $\eta e^{\widetilde U_k/2}\in H^1$. 
			Applying \eqref{eq:rescaled-stability} to
			$f=\eta e^{\widetilde U_k/2}$ gives
			\begin{align}\label{eq inte ineq}
			\int_\Omega\eta^2W_k^2
			\leq
			\int_\Omega W_k|\eta'|^2
			+\int_\Omega\eta\eta'W_k\widetilde U_{k,s}
			+\frac14\int_\Omega
			\eta^2W_k|\nabla_k\widetilde U_k|^2.
			\end{align}
			
			Fix $k$. For $\varepsilon>0$ sufficiently small, the disks
			$B_\varepsilon(\widetilde q_{j,k})$ are mutually disjoint
			and contained in $\Omega$. Set
			\[
			\Omega_{\varepsilon,k}
			=
			\Omega\setminus\bigcup_{j=1}^{m_k}
			\overline{B_\varepsilon(\widetilde q_{j,k})}.
			\]
			On $\Omega_{\varepsilon,k}$,
			$\Delta_k\widetilde U_k+W_k=0$. Hence, multiplying by
			$\eta^2W_k$ and integrating by parts, we obtain
			\[
			\begin{aligned}
				\int_{\Omega_{\varepsilon,k}}\eta^2W_k^2
				={}&
				2\int_{\Omega_{\varepsilon,k}}
				\eta\eta'W_k\widetilde U_{k,s}
				+\int_{\Omega_{\varepsilon,k}}
				\eta^2W_k|\nabla_k\widetilde U_k|^2\\
				&-
				\int_{\partial\Omega_{\varepsilon,k}}
				\eta^2W_k\partial_{\nu_k}\widetilde U_k\,\ud\sigma ,
			\end{aligned}
			\]
			where, if $\nu=(\nu_s,\nu_\theta)$ denotes the outward unit
			normal,
			\[
			\partial_{\nu_k}\widetilde U_k
			:=
			\nu_s\widetilde U_{k,s}
			+\ell_k^2\nu_\theta\widetilde U_{k,\theta}.
			\]
			then from \eqref{eq: asy est}, there holds $|\partial_{\nu_k}\widetilde U_k|=O_{j,k}(r^{-1})$.
			Therefore the boundary integral over
			$\partial B_\varepsilon(\widetilde q_{j,k})$ is
			$O_{j,k}(\varepsilon^{\beta_{j,k}/(2\pi)})$.
			Since $m_k$ is finite and every $\beta_{j,k}>0$, the sum
			of the boundary terms tends to zero as
			$\varepsilon\to0$ with $k$ fixed. Letting
			$\varepsilon\to0$ gives
			\begin{align}\label{eq inte indend}
			\int_\Omega\eta^2W_k^2
			=
			2\int_\Omega\eta\eta'W_k\widetilde U_{k,s}
			+\int_\Omega
			\eta^2W_k|\nabla_k\widetilde U_k|^2.
			\end{align}
		Using \eqref{eq inte ineq} and \eqref{eq inte indend}, and canceling the terms containing $\eta\eta'W_k\widetilde U_{k,s}$, we obtain
			\begin{equation}\label{eq:sqrt-caccioppoli}
				\int_\Omega\eta^2
				\left(
				W_k^2+\frac12W_k|\nabla_k\widetilde U_k|^2
				\right)
				\leq
				2\int_\Omega W_k|\eta'|^2.
			\end{equation}
			Let $d(s)=\min\{s,1-s\}$.
			Approximating $d$ by compactly supported cutoffs in
			\eqref{eq:sqrt-caccioppoli}, and also using cutoffs
			equal to one on compact subintervals, it follows from \eqref{eq:wkin} that
			\begin{equation}\label{eq:weighted-potential-bound}
				\|dW_k\|_{L^2(\Omega)}\leq C,
				\qquad
				\int_{K\times\mathbb S^1}
				\left(
				|\nabla_k\sqrt{W_k}|^2+W_k^2
				\right)\leq C_K
				\quad(K\Subset(0,1)).
			\end{equation}
			Here and below, the constants $C$ and $C_K$ are independent
			of $k$ and of the singularity data.
			
			{\bf Step 2: Uniform bound for $E_k(\widetilde{\psi}_k)$.}
			We use the following global interpolation estimate:
			\begin{equation}\label{eq:weighted-interpolation}
				\left\|\frac{f^2}{d}\right\|_{L^2(\Omega)}
				\leq
				C\|f\|_{L^2(\Omega)}^{1/2}E_k(f)^{3/4}
				+C\ell_k^{-1/2}E_k(f),
				\qquad f\in H^1_0(\Omega).
			\end{equation}
			To verify it, first take $f$ smooth and put
			$X(s)=\|f(s,\cdot)\|_{L^2(\mathbb S^1)}$ and
			$Y(s)=\|f_\theta(s,\cdot)\|_{L^2(\mathbb S^1)}$.
			Since $X\in H^1_0(0,1)$ and
			$|X'|\leq\|f_s(s,\cdot)\|_{L^2(\mathbb S^1)}$,
			Hardy's inequality and the fundamental theorem of
			calculus give
			\begin{align}\label{eq X ineq}
			&\int_0^1\frac{X^2}{d^2}\,ds
			\leq C\int_0^1(X'(s))^2\ud s
			\leq CE_k(f),\\
			&X(s)^2
			\leq d(s)\int_0^1(X'(s))^2\ud s
			\leq d(s)E_k(f),\label{eq X ineq0}
			\end{align}
			and
			\begin{align}\label{eq X ineq-1}
				\|X\|_\infty^2\leq C\|X\|_{L^2}\|X'\|_{L^2}
				\leq C\|f\|_{L^2(\Omega)}E_k(f)^{1/2}.
			\end{align}
			The  Gagliardo-Nirenberg inequality on $\mathbb S^1$ gives
            \begin{align*}
                \|f(s,\cdot)\|^2_{L^{\infty}(\mathbb S^1)} \leq C\left(\|f(s,\cdot)\|^2_{L^{2}(\mathbb S^1)}+\|f(s,\cdot)\|_{L^{2}(\mathbb S^1)}\|f_{\theta}(s,\cdot)\|_{L^{2}(\mathbb S^1)}  \right),
            \end{align*}
            then
            \begin{align}\label{eq: f ineq}
               \int_{\mathbb S^1}|f|^4\leq \|f(s,\cdot)\|^2_{L^{\infty}(\mathbb S^1)} \|f(s,\cdot)\|^2_{L^{2}(\mathbb S^1)}\leq C(X^4+X^3Y). 
            \end{align}
			Consequently, using \eqref{eq: f ineq} and \eqref{eq X ineq}-\eqref{eq X ineq0},
			\[
			\begin{aligned}
				\int_\Omega\frac{|f|^4}{d^2}&=\int_0^1\frac{1}{d^2(s)}ds\int_{\mathbb S^1}|f|^4\\
                &\leq
				C\|X\|_\infty^2\int_0^1\frac{X^2}{d^2}\,ds
				+CE_k(f)\int_0^1\frac{XY}{d}\,ds\\
				&\leq
				C\|f\|_{L^2(\Omega)}E_k(f)^{3/2}
				+CE_k(f)\left(\int_0^1\frac{X^2}{d^2}\,ds \right)^{\frac{1}{2}} \left(\int_0^1Y^2\,ds \right)^{\frac{1}{2}} \\
                &\leq
				C\|f\|_{L^2(\Omega)}E_k(f)^{3/2}
				+C\ell_k^{-1}E_k(f)^2.
			\end{aligned}
			\]
			Taking square roots and using density proves
			\eqref{eq:weighted-interpolation}.
			
			We now apply this estimate only to $\widetilde{\psi}_k$.
			By \eqref{eq:rescaled-stability},
			\eqref{eq:weighted-potential-bound}, and the normalization $\|\widetilde{\psi}_k\|_{L^2(\Omega)}=1$,
			\[
			\begin{aligned}
				E_k(\widetilde{\psi}_k)
				&=\int_\Omega W_k\widetilde{\psi}_k^2
				\leq
				\|dW_k\|_{L^2(\Omega)}
				\left\|\frac{\widetilde{\psi}_k^2}{d}\right\|_{L^2(\Omega)}\\
				&\leq
				C E_k(\widetilde{\psi}_k)^{3/4}
				+C\ell_k^{-1/2}E_k(\widetilde{\psi}_k).
			\end{aligned}
			\]
			For sufficiently large $k$, the last term can be
			absorbed into the left-hand side. Hence
			\begin{equation}\label{eq:jacobi-energy-bound}
				\int_{\Omega}|\partial_s\widetilde{\psi}_k|^2+\ell_k^2|\partial_\theta\widetilde{\psi}_k|^2=E_k(\widetilde{\psi}_k)\leq C.
			\end{equation}
			
			{\bf Step 3: The one-dimensional limit.}
			By \eqref{eq:weighted-potential-bound},
			$\sqrt{W_k}$ is bounded in $H^1_{\mathrm{loc}}(\Omega)$.
			Together with \eqref{eq:jacobi-energy-bound}, compactness
			on the fixed cylinder therefore gives, after passing
			to a subsequence,
			\[
			\begin{gathered}
				\widetilde{\psi}_k\rightharpoonup p(s)
				\quad\text{in }H^1_0(\Omega),
				\qquad
				\widetilde{\psi}_k\to p(s)
				\quad\text{in }L^2(\Omega),\\
			\sqrt{W_k}\rightharpoonup\sqrt m,
				\quad\text{in }H^1_{\mathrm{loc}}(\Omega)\qquad	\sqrt{W_k}\to\sqrt m
				\quad\text{in }L^4_{\mathrm{loc}}(\Omega),
			\end{gathered}
			\]
			and hence
			\[
			W_k\to m(s)
			\quad\text{in }L^2_{\mathrm{loc}}(\Omega),
			\qquad
			\sqrt m\in H^1_{\mathrm{loc}}(0,1).
			\]
			Both limits are independent of $\theta$, since
			$\partial_\theta\widetilde{\psi}_k\to0$ in $L^2(\Omega)$ and
			$\partial_\theta\sqrt{W_k}\to0$ in
			$L^2_{\mathrm{loc}}(\Omega)$.
			
			Passing to the Jacobi equation against test functions
			depending only on $s$, we obtain
			\[
			p''+mp=0,
			\qquad
			p\in H^1_0(0,1),\quad p\geq0,\quad p\not\equiv0.
			\]
			Here $p\not\equiv0$ follows from the normalization and
			strong $L^2$ convergence.
			Since $m\geq0$, the function $p$ is concave and hence
			positive in $(0,1)$.
			In particular, $m\not\equiv0$; moreover,
			$m\in L^1(0,1)$ by Fatou's lemma.
			
			It remains to identify the equation satisfied by the limiting
			potential $m$. Using \eqref{eq: asy est} and applying the similar arguments for \eqref{eq inte indend}, we can get
			\[
			\Delta_kW_k
			=
			4|\nabla_k\sqrt{W_k}|^2-W_k^2
			\qquad\text{in }\Omega
			\]
			in the sense of distributions. Denote circular averaging by
			\[
			\langle F\rangle
			:=
			\frac1{2\pi}\int_{\mathbb S^1}F\,d\theta .
			\]
			Integrating over $\mathbb S^1$ gives
			\[
			\langle W_k\rangle''
			=
			4\left\langle|\nabla_k\sqrt{W_k}|^2\right\rangle
			-\langle W_k^2\rangle .
			\]
			Let $\zeta\in C_c^\infty(0,1)$ be nonnegative. Then
			\[
			\int_0^1\langle W_k\rangle\zeta''
			=
			4\int_0^1
			\left\langle|\nabla_k\sqrt{W_k}|^2\right\rangle\zeta
			-
			\int_0^1\langle W_k^2\rangle\zeta .
			\]
			Since $W_k\to m$ strongly in $L^2_{\rm loc}$, 
			$\sqrt{W_k}\rightharpoonup\sqrt m$ weakly in
			$H^1_{\rm loc}$ and $\int_0^1\left\langle|\nabla_k\sqrt{W_k}|^2\right\rangle\zeta\geq \int_0^1 \left\langle|\partial_s\sqrt{W_k}|^2\right\rangle\zeta $, weak lower semicontinuity gives
			\[
			\int_0^1m\zeta''
			\geq
			4\int_0^1|(\sqrt m)'|^2\zeta
			-
			\int_0^1m^2\zeta .
			\]
			Hence
			\begin{align}\label{eq:limiting-potential-inequality}
				m''\geq4|(\sqrt m)'|^2-m^2
			\end{align}
			in the sense of distributions.
			
			Since $\sqrt m\in H^1_{\rm loc}(0,1)$, the one-dimensional
			Sobolev embedding implies that $m$ is continuous and locally
			bounded. In particular, on every $K\Subset(0,1)$,
			\eqref{eq:limiting-potential-inequality} gives
			$m''\geq-m^2\geq-C_K$.
			Thus $m$ is locally semiconvex and hence locally Lipschitz,
			i.e.,
			$m\in W^{1,\infty}_{\mathrm{loc}}(0,1)$.
			On each component of $\{m>0\}$,
			\eqref{eq:limiting-potential-inequality} gives
			\[
			(\log m)''+m\geq0.
			\]
			Let $J$ be a connected component of $\{m>0\}$. We claim that $J$ cannot have an endpoint in $(0,1)$. Suppose, for contradiction, that $a\in(0,1)$ is an endpoint of $J$. Since $m$ is continuous and $m>0$ in $J$, we have $m(a)=0$. As $m$ is locally bounded, there exist $\varepsilon>0$ and $C>0$ such that $m\leq C$ on $J\cap(a-\varepsilon,a+\varepsilon)$. Since $(\log m)''+m\geq0$ in $J$, we obtain $(\log m)''\geq-C$ there. Hence
\[
h(s):=\log m(s)+\frac{C}{2}s^2
\]
is convex on $J\cap(a-\varepsilon,a+\varepsilon)$. Fix $s_0$ in this interval. By convexity, there exists $q\in\mathbb R$ such that
$h(s)\geq h(s_0)+q(s-s_0)$ for $s$ sufficiently close to $a$ within $J$. Thus $\log m$ is bounded from below near $a$. On the other hand, by the continuity of $m$ and $m(a)=0$, we have $\log m(s)\to-\infty$ as $s\to a$ within $J$, a contradiction. Therefore no connected component of $\{m>0\}$ has an endpoint in $(0,1)$. Since $m\not\equiv0$, it follows that
\[
m>0\qquad\text{in }(0,1).
\]
			
			Thus $v=\log m$ belongs to
			$W^{1,\infty}_{\mathrm{loc}}(0,1)$ and satisfies
			$v''+e^v\geq0$ in the sense of distributions, while
			$p''+e^vp=0$ with $p>0$ and $p\in H^1_0(0,1)$.
			Using the nonnegative-measure extension of
			Lemma~\ref{lem:one-dimensional-mass}, we obtain
			\[
			\int_0^1m\,ds\geq8.
			\]
			On the other hand, lower semi-continuity
			and \eqref{eq:mass-deficit} give
			\[
			8
			\leq\int_0^1m\,ds
			\leq\frac1{2\pi}\liminf_{k\to\infty}\int_\Omega W_k
			\leq8-\frac{\delta}{2\pi},
			\]
			a contradiction.
			Together with $B(c)\leq8c$, this proves
			$B(c)=8c+o(c)$ as $c\downarrow0$.
		\end{proof}
		
		\section{Proof of uniqueness}\label{Sec 5}
		The main aim of this section is to prove Theorem \ref{thm:main} and Theorem \ref{thm:intro-asymptotic-uniqueness}. The only remaining step is to establish the relationship between the capacity $c$ and the quantity $\lambda_1(\T)|\T|$.

\begin{lemma}
	\label{lem:capacity-sum-analytic}
	Let \(\mathbb T=\mathbb C/\Lambda\) be a flat two-dimensional torus, and let
	\(\Gamma_0,\ldots,\Gamma_{m-1}\subset\mathbb T\) be pairwise disjoint
	parallel essential simple closed curves such that
	\[
	\mathbb T\setminus\bigcup_{j=0}^{m-1}\Gamma_j
	=
	\bigsqcup_{j=1}^{m}A_j,
	\]
	where \(A_j\) is the annulus bounded by \(\Gamma_{j-1}\) and \(\Gamma_j\),
	with the indices understood modulo \(m\). Set
	\(c_j:=\operatorname{Cap}(A_j)\). Then
	\begin{equation}
		\label{eq:reciprocal-capacity-sum}
		\sum_{j=1}^{m}\frac{1}{c_j}
		\leq
		\frac{4\pi^2}{\lambda_1(\mathbb T)|\mathbb T|}.
	\end{equation}
	In particular,
	\begin{equation}
		\label{eq:capacity-sum-analytic}
		\sum_{j=1}^{m}c_j
		\geq
		\frac{m^2}{4\pi^2}
		\lambda_1(\mathbb T)|\mathbb T|.
	\end{equation}
\end{lemma}

\begin{proof}
	Since the curves \(\Gamma_j\) are parallel essential simple closed curves,
	after choosing a common orientation their common homotopy class is
	represented by a primitive lattice vector
	\(\gamma_{v}\in\Lambda\).
	Choose \(\gamma_{h}\in\Lambda\) such that
	\(\Lambda=\mathbb Z\gamma_{v}\oplus
	\mathbb Z\gamma_{h}\), and set
	\[
	J:=\left|\det(\gamma_{v},\gamma_{h})\right|
	=|\mathbb T|.
	\]

	For each \(j\), let \(h_j\) be the capacitary potential of \(A_j\),
	normalized by \(h_j=0\) on \(\Gamma_{j-1}\) and \(h_j=1\) on
	\(\Gamma_j\). Thus
	\[
	\int_{A_j}|\nabla h_j|^2\,dA=c_j.
	\]
	Put
	\[
	R:=\sum_{j=1}^{m}\frac1{c_j},
	\qquad
	a_j:=\frac{c_j^{-1}}{R},
	\]
	so that \(a_j>0\) and \(\sum_{j=1}^{m}a_j=1\).

	Let \(\pi:\mathbb R^2\to\mathbb T\) be the universal covering map.
	Choose an ordered sequence of connected lifts
	\(\{\widetilde\Gamma_k\}_{k\in\mathbb Z}\) such that
	\[
	\widetilde\Gamma_{k+m}
	=
	\widetilde\Gamma_k+\gamma_{h},
	\qquad
	\widetilde\Gamma_k+\gamma_{v}
	=
	\widetilde\Gamma_k.
	\]
	Let \(\widetilde A_k\) be the strip bounded by
	\(\widetilde\Gamma_{k-1}\) and \(\widetilde\Gamma_k\), and let
	\(\widetilde h_k\) denote the lift of the corresponding \(h_j\), with
	\(\widetilde h_k=0\) on \(\widetilde\Gamma_{k-1}\) and
	\(\widetilde h_k=1\) on \(\widetilde\Gamma_k\). Extend \(a_j\)
	periodically by \(a_{k+m}=a_k\).

	Let \(s_0=0\) and define \(s_k-s_{k-1}=a_k\). Since
	\(\sum_{j=1}^{m}a_j=1\), we have \(s_{k+m}=s_k+1\). Define
	\[
	\widetilde H(x)
	=
	s_{k-1}+a_k\widetilde h_k(x),
	\qquad x\in\widetilde A_k.
	\]
	The traces agree on every interface \(\widetilde\Gamma_k\), hence
	\(\widetilde H\in H^1_{\mathrm{loc}}(\mathbb R^2)\). Moreover,
	\begin{equation}
		\label{eq:Htilde-weighted-quasiperiodic}
		\widetilde H(x+\gamma_{v})=\widetilde H(x),
		\qquad
		\widetilde H(x+\gamma_{h})=\widetilde H(x)+1.
	\end{equation}
	Since \(\nabla\widetilde H=a_j\nabla h_j\) on the lift of \(A_j\),
	for any fundamental parallelogram \(P\) of \(\Lambda\),
	\begin{equation}
		\label{eq:weighted-height-energy}
		\int_P|\nabla\widetilde H|^2\,dx
		=
		\sum_{j=1}^{m}a_j^2c_j
		=
		\frac1R.
	\end{equation}

	We next estimate this energy from below. Introduce lattice coordinates
	\(F(s,t)=s\gamma_{v}+t\gamma_{h}\) and set
	\(H(s,t):=\widetilde H(F(s,t))\). By
	\eqref{eq:Htilde-weighted-quasiperiodic},
	\[H(s+1,t)=H(s,t),\qquad H(s,t+1)=H(s,t)+1.\]
	Put \(A:=|\gamma_{v}|^2\),
	\(B:=|\gamma_{h}|^2\), and
	\(C:=\gamma_{v}\cdot\gamma_{h}\). Since
	\(AB-C^2=J^2\), we have
	\begin{align*}
		\int_P|\nabla\widetilde H|^2\,dx
		&=
		\frac1J\int_0^1\int_0^1
		\left(BH_s^2-2CH_sH_t+AH_t^2\right)\,ds\,dt\\
		&=
		\frac1J\int_0^1\int_0^1
		\left[
		\frac{J^2}{A}H_s^2
		+
		A\left(H_t-\frac CAH_s\right)^2
		\right]\,ds\,dt\\
		&\geq
		\frac{A}{J}
		\left[
		\int_0^1\int_0^1
		\left(H_t-\frac CAH_s\right)\,ds\,dt
		\right]^2
		=
		\frac{A}{J},
	\end{align*}
	where the last equality follows from the quasi-periodicity of \(H\).
	Combining this with \eqref{eq:weighted-height-energy} yields
	\begin{equation}
		\label{eq:reciprocal-capacity-homotopy}
		R
		\leq
		\frac{|\mathbb T|}{|\gamma_{v}|^2}.
	\end{equation}

	It remains to estimate the right-hand side by \(\lambda_1(\mathbb T)\).
	Consider the well-defined zero-average function
	\(\varphi(F(s,t))=\cos(2\pi t)\) on \(\mathbb T\). Since
	\(\varphi_s=0\) and \(\varphi_t=-2\pi\sin(2\pi t)\), the Rayleigh
	characterization gives
	\[
	\lambda_1(\mathbb T)
	\leq
	\frac{\displaystyle\int_{\mathbb T}|\nabla\varphi|^2\,dA}
	{\displaystyle\int_{\mathbb T}\varphi^2\,dA}
	=
	\frac{4\pi^2|\gamma_{v}|^2}{|\mathbb T|^2}.
	\]
	Hence
	\[
	\frac{|\mathbb T|}{|\gamma_{v}|^2}
	\leq
	\frac{4\pi^2}{\lambda_1(\mathbb T)|\mathbb T|}.
	\]
	Together with \eqref{eq:reciprocal-capacity-homotopy}, this proves
	\eqref{eq:reciprocal-capacity-sum}.

	Finally, by the Cauchy--Schwarz inequality,
	\[
	m^2
	\leq
	\left(\sum_{j=1}^{m}c_j\right)
	\left(\sum_{j=1}^{m}\frac1{c_j}\right).
	\]
	Using \eqref{eq:reciprocal-capacity-sum}, we obtain
	\[
	\sum_{j=1}^{m}c_j
	\geq
	\frac{m^2}{4\pi^2}
	\lambda_1(\mathbb T)|\mathbb T|,
	\]
	which proves \eqref{eq:capacity-sum-analytic}.
\end{proof}

\begin{proof}[\bf Proof of Theorem \ref{thm:main}]
 Let \(u\) be the corresponding solution and \(\phi\not\equiv0\) the Jacobi field:
\[
   \Delta\phi+e^u\phi=0.
\]
The orthogonality relation at first degeneracy implies \(\int_\T e^u\phi=0\), hence \(\phi\) changes sign. The linearized sphere-covering inequality excludes simply connected nodal domains below total mass \(4\pi\). Therefore the nodal domains are parallel annuli
\[
   A_1,\ldots,A_m,
   \qquad m\ge2.
\]
On each \(A_j\), the function \(|\phi|\) is positive, vanishes on \(\partial A_j\), and satisfies
\[
   \Delta |\phi|+e^u|\phi|=0.
\]
Also \(\Delta u+e^u\ge0\) on each \(A_j\) in the sense of distributions.

Put
\[
   M_j:=\int_{A_j}e^u,
   \qquad
   c_j:=\Cap(A_j).
\]
By the conformal invariance of the equation, mass and capacity, Theorem \ref{thm:local} gives
\[
   M_j> \frac{4\pi \kappa c_j}{2+ \kappa c_j}.
\]
Since \(N\rho=\int_\T e^u=\sum_jM_j\) and Jensen inequality,
\[
   N\rho> 4\pi\sum_{j=1}^m\frac{1}{\frac{2}{\kappa c_j}+1}\geq \frac{4\pi m}{\frac{1}{m}\sum_{j=1}^m \frac{2}{\kappa c_j}+1}\geq \frac{4\pi m}{\frac{8\pi^2}{m\kappa\lambda_1(\T)|\T|}+1}\geq \frac{8\pi \kappa\lambda_1(\T)|\T|}{4\pi^2+\kappa\lambda_1(\T)|\T| }.
\]
Therefore, if 
\begin{align}\label{eqrhor}
    \rho\leq \frac{1}{N}\min\left\{4\pi,\frac{8\pi \kappa\lambda_1(\T)|\T|}{4\pi^2+\kappa\lambda_1(\T)|\T| } \right\},
\end{align}
then every solution in this parameter range is nondegenerate.
To conclude uniqueness, fix the singular points $p_1,\cdots, p_N$ and the singular masses $\beta_1,\cdots,\beta_N$ such that $\rho=\frac{1}{N}\sum\beta_j$ satisfies \eqref{eqrhor}, and consider
\[
\Delta u_t+e^{u_t}
=t\sum_{j=1}^N\beta_j\delta_{p_j},
\qquad 0<t\le1.
\]
Since $N\rho\leq 4\pi<8\pi$, the compactness
result of \cite{BT} excludes blow-up on
every compact subinterval of $(0,1]$.
Together with nondegeneracy and the implicit function theorem,
this implies that the number of solutions is constant for $t\in (0,1]$.
Therefore, the uniqueness for small $t$ proved in Section~\ref{Sec 2}
therefore gives uniqueness for $t=1$.
\end{proof}

\begin{proof}[\bf Proof of Theorem~\ref{thm:intro-asymptotic-uniqueness}]
Fix $\varepsilon\in(0,1)$. By Theorem~\ref{thm:small-capacity-asymptotic},
there exists $c_\varepsilon>0$ such that
\[
 B(c)\geq 8\left(1-\frac{\varepsilon}{2}\right)c,
 \qquad 0<c\leq c_\varepsilon.
\]

We first observe that, if $\delta_\varepsilon$ is sufficiently
small, then $c_j<c_\varepsilon$ for every $j$. Indeed, by
Theorem \ref{thm:local},
\[
 B(c)\geq\frac{4\pi \kappa c}{2+\kappa c},
\]
so if $c_j\geq c_\varepsilon$ for some $j$, then
\[
 N\rho=\sum_{i=1}^mM_i
 \geq M_j
 \geq\frac{4\pi \kappa c_\varepsilon}{2+\kappa c_\varepsilon}.
\]
This contradicts
\[
 N\rho
 \leq\frac{8(1-\varepsilon)}{\pi^2}
 \lam(\T)|\T|
\]
once $\lam(\T)|\T|\leq\delta_\varepsilon$ and
$\delta_\varepsilon$ is chosen sufficiently small.
Decreasing $\delta_\varepsilon$ further if necessary, we may
also assume $N\rho<4\pi$, so that the nodal-domain argument
used above applies. Hence Theorem~\ref{thm:small-capacity-asymptotic} and
Lemma \ref{lem:capacity-sum-analytic} give
\[
 \begin{aligned}
 N\rho
 &=\sum_{j=1}^mM_j
 \geq
 8\left(1-\frac{\varepsilon}{2}\right)
 \sum_{j=1}^mc_j\\
 &\geq
 8\left(1-\frac{\varepsilon}{2}\right)
 \frac{m^2}{4\pi^2}\lam(\T)|\T|
 \geq
 \frac{8(1-\varepsilon/2)}{\pi^2}
 \lam(\T)|\T|.
 \end{aligned}
\]
This contradicts
\[
 N\rho
 \leq
 \frac{8(1-\varepsilon)}{\pi^2}
 \lam(\T)|\T|.
\]
This completes the proof.
\end{proof}

\section{Non-uniqueness of solutions}\label{Sec 6}

This final section is devoted to the proofs of Theorems \ref{thm:intro-nonuniqueness}, \ref{thm 1.5} and more multiplicity results.
Let $\tau \in \mathbb{H}=\left \{  \tau\in\mathbb C|\operatorname{Im}\tau>0\right \}$, $\Lambda_{\tau}=\mathbb{Z}+\mathbb{Z}\tau$, and denote
$$\omega_{0}=0,\quad\omega_{1}=1,\quad\omega_{2}=\tau,\quad\omega_{3}=1+\tau.$$Let $E_{\tau}:=\mathbb{C}/\Lambda_{\tau}$ be a normalized flat torus and $E_{\tau}[2]:=\{ \frac{\omega_{k}}{2}|k=0,1,2,3\}+\Lambda
_{\tau}$.  
The Green function $G(z,w)=G(z,w;\tau)$ of the flat torus $E_{\tau}$ is the unique function that satisfies
\[
-\Delta_z G(z, w)=\delta_{w}-\frac{1}{\left \vert E_{\tau}\right \vert }\text{
\ on }E_{\tau},\quad
\int_{E_{\tau}}G(z,w)\mathrm{d}A_z=0.
\]
By the translation invariance of $\Delta_z$, we have $G(z,w)=G(z-w,0)$ and it is enough to consider the Green function
$G(z):=G(z,0)$, then $G(z,w)=G(z-w)$.
Clearly $G(z)$ is an even function on $E_{\tau}$ with the only
singularity at $0$.

Firstly, we consider the precise formula of $\lambda_1(\T)$ for  flat torus $\T=E_{\tau}$. Let $\tau=a+bi$, recall the dual lattice in \eqref{Intro lambda}, $\xi\in \Lambda_{\tau}^*$ is equivalent to $\xi\cdot (1,0), \xi\cdot (a,b)\in 2\pi\mathbb{Z}$, then  there exists $(m,n)\in \mathbb{Z}\times\mathbb{Z}$ such that 
\begin{align*}
    \xi_1=2\pi m,\qquad \xi_1a+\xi_2 b=2\pi n.
\end{align*}
Using \eqref{Intro lambda_1}, 
then we can get $$\lambda_1(E_{\tau})=\frac{4\pi^2}{(\mathrm{Im}\tau)^2}\min_{m,n\not=(0,0)}|n-m\tau|^2. $$
If  $\mathrm{Im}\tau >\frac{\pi}{2} $, then 
\begin{align*}
    \lambda_1(E_{\tau})=\frac{4\pi^2}{(\mathrm{Im}\tau)^2}, \quad \lambda_1(E_{\tau})|E_{\tau}|=\frac{4\pi^2}{\mathrm{Im}\tau}<8\pi.
\end{align*}
We also introduce a special translation 
\begin{align}\label{Sigma tau}
    \sigma_{\tau}(z)=z+\frac{\tau}{2},\qquad \sigma_\tau(\frac{\tau}{4})=-\frac{\tau}{4} \quad\text{in }E_{\tau}.
\end{align}
A function $u$ is called $\sigma_{\tau}$ invariant if and only if $u\circ \sigma_{\tau}=u$ in $E_{\tau}$.
Motivated by the variational instability mechanism for the smooth mean field equation on flat tori, developed in Ricciardi–Tarantello \cite{RicciardiTarantello98}, Lin–Lucia \cite{LinLucia07}, and Gu–Gui–Hu–Li \cite{GGHL}, we adapt the first-eigenmode argument to a symmetric pair of positive singularities.

\begin{theorem}\label{thm:nonunique}
Suppose that \(\mathrm{Im}\,\tau>\frac{\pi}{2}\) and $
\frac{1}{2}\lam(E_{\tau})|E_{\tau}|<\rho<4\pi$ . Then one can choose the point \(p=\frac{\tau}{4}\in E_{\tau}\) such that the equation
\begin{align}\label{eq: p-p}
	\Delta u+e^u=\rho(\delta_p+\delta_{-p})
	\qquad \text{on } E_{\tau}
\end{align}
admits at least two solutions that are not $\sigma_{\tau}$-invariant and a unique $\sigma_{\tau}$-invariant solution.
\end{theorem}

\begin{proof}
Choose $
   p=\frac{\tau}{4}
$ and define 
the translation \(\sigma_{\tau}(z)=z+\frac{\tau}{2}\) interchanges \(p\) and \(-p\).
Put
\[
   S(z)=-\rho G(z,p)-\rho G(z,-p),
   \qquad
   h=e^S.
\]
Then \eqref{eq: p-p} is equivalent to
\begin{equation}\label{eq:regularized}
   -\Delta v=2\rho\left(\frac{he^v}{\int_{E_{\tau}} he^v\,\ud A}-\frac1{|E_{\tau}|}\right),
   \qquad
   \int_{E_{\tau}} v\,\ud A=0,
\end{equation}
with \(u=S+v+c\). Equation \eqref{eq:regularized} is the Euler--Lagrange equation of
\[
   J(v)=\frac12\int_{E_{\tau}} |\nabla v|^2\,\ud A
        -2\rho\log\int_{E_{\tau}} he^v\,\ud A
\]
on the zero-average space. Since \(\rho<4\pi\), the Moser--Trudinger inequality gives coercivity, so \(J\) has a global minimizer. The weight \(h\) is \(\sigma_{\tau}\)-invariant.

Suppose that a global minimizer \(v\) is \(\sigma_{\tau}\)-invariant. Let
\[
   \ud P_v=\frac{he^v}{\int_\T he^v\,\ud A}\,\ud A.
\]
Then \(P_v\) is \(\sigma_{\tau}\)-invariant. Set $\xi_0=\frac{2\pi i}{\mathrm{Im}\tau} $ and 
\[
   \varphi_1(z)=\cos(\mathrm{Re}(\xi_0\bar z)),
   \qquad
   \varphi_2(z)=\sin(\mathrm{Re}(\xi_0\bar z)),
\]
where we use the complex variable $z=x+iy$.
Since \(\varphi_i\circ\sigma_{\tau}=-\varphi_i\), we have \(\int_{E_{\tau}} \varphi_i\,\ud P_v=0\). The second variation is
\[
   D^2J_v(\varphi)=\int_{E_{\tau}} |\nabla\varphi|^2\,\ud A
   -2\rho\left(\int_{E_{\tau}}\varphi^2\,\ud P_v-\left(\int_{ E_{\tau}}\varphi\,\ud P_v\right)^2\right).
\]
Since $v$ is the global minimizer, we have \( D^2J_v\ge0\). But
\[
   \varphi_1^2+\varphi_2^2=1,
   \qquad
   |\nabla\varphi_1|^2+|\nabla\varphi_2|^2=|\xi_0|^2=\lam(E_{\tau}),
\]
and hence
\[
   D^2J_v(\varphi_1)+D^2J_v(\varphi_2)=\lam(E_{\tau})|E_{\tau}|-2\rho<0,
\]
a contradiction. Thus no global minimizer is \(\sigma_{\tau}\)-invariant. If \(v\) is a minimizer, then \(v\circ\sigma_{\tau}\) is another distinct minimizer, and the corresponding functions \(u=S+v+c\) and \(u_2=S+v\circ\sigma_{\tau}+c\) are two distinct solutions of \eqref{eq: p-p}.

Finally, \eqref{eq: p-p} has a third solution that is \(\sigma_{\tau}\)-invariant.  Let $E_{\tau}'$ be the flat torus generated by $\tau/2$ and $1$ , and let $u_3(x)$ be a solution of
\begin{equation}\label{qeeq}\Delta u+e^u=\rho\delta_p
	\qquad \text{on } E_{\tau}',\end{equation}
    then $u(z+\tau/2)=u(z)$. 
    Since $p=-p$ in $E_{\tau}'$ but $p\neq -p$ in $E_{\tau}$, we see that $u_3(z)$ is a \(\sigma_{\tau}\)-invariant solution of \eqref{eq: p-p}. 
    
    Suppose there exists another solution $u_3'$ satisfying $u_3'\circ\sigma_{\tau}=u_3'$, then $u_3'$ is well-defined on $E'_{\tau}$ and hence a solution of \eqref{qeeq}. Since $\rho<4\pi$, as mentioned in \eqref{eq:one-rho}, it follows from \cite[Theorem 1.4]{LinWang17} that the solution of \eqref{qeeq} is unique, so $u_3'=u_3$. Therefore, the $\sigma_{\tau}$-invariant solution is unique.
\end{proof}

By the similar technique, we can construct multiple solutions of \eqref{eq:asymptotic-nonunique} for almost sharp $\rho$.

\begin{theorem}\label{thm:asymptotic-nonunique}
For every $\varepsilon>0$, there exists $b_\varepsilon>0$ such that,
if $\tau=\mathrm{i}b$ and $b\geq b_\varepsilon$, then the equation
\begin{equation}\label{eq:asymptotic-nonunique}
 \Delta u+e^u=\frac{16(1+\varepsilon)}{b}
 (\delta_p+\delta_{-p})
 \qquad\text{on }E_\tau
\end{equation}
admits at least three distinct solutions, where $p=\tau/4\in E_\tau$.
\end{theorem}

\begin{proof}
Fix $\varepsilon>0$ and put $\kappa=16(1+\varepsilon)$ (In this proof, the notation $\kappa$ is not the one defined in \eqref{kappa}) and
$\rho=\kappa/b$. We use the notation $S$, $h$, $J$, and
$\sigma_\tau$ from the proof of Theorem~\ref{thm:nonunique},
with $p=\tau/4$. For sufficiently large $b$, $\rho<4\pi$.
The existence arguments in that proof, which require only
$\rho<4\pi$, give a global minimizer of $J$ and a
$\sigma_\tau$-invariant solution.
It therefore suffices to prove that no global minimizer
is $\sigma_\tau$-invariant when $b$ is sufficiently large.

Write $z=x+\tau s=x+\mathrm{i}bs$, where
$(x,s)\in Q:=(\mathbb R/\mathbb Z)^2$, and set
\[
 S_b(x,s):=S(x+\tau s),
 \qquad h_b=e^{S_b}.
\]
The area element is $\ud A=b\,\ud x\ud s$.
The normalization of the Green function gives
\begin{equation}\label{eq:Sb-equation}
 \left(b^2\partial_x^2+\partial_s^2\right)S_b
 =\kappa\bigl(\delta_{(0,1/4)}+\delta_{(0,-1/4)}\bigr)-2\kappa,
 \qquad \int_QS_b\,\ud x\ud s=0.
\end{equation}
Define
\[
S_0(s)=
\begin{cases}
\displaystyle
\kappa\left(\frac{1}{48}-\left(s+\frac12\right)^2\right),
& -\frac12\leq s<-\frac14,\\[2mm]
\displaystyle
\kappa\left(\frac{1}{48}-s^2\right),
& -\frac14\leq s\leq\frac14,\\[2mm]
\displaystyle
\kappa\left(\frac{1}{48}-\left(s-\frac12\right)^2\right),
& \frac14<s\leq\frac12,
\end{cases}
\]
and extend it $1$-periodically. This function is also
$1/2$-periodic, has zero average, and satisfies
$S_0''=\kappa(\delta_{1/4}+\delta_{-1/4})-2\kappa$.
Set $R_b=S_b-S_0$. Then $R_b$ satisfies
\[
 \left(b^2\partial_x^2+\partial_s^2\right)R_b
 =
 \kappa(\delta_0(x)-1)
 \left(\delta_{1/4}(s)+\delta_{-1/4}(s)\right),
\]
and has zero average.
Writing
\[
 R_b(x,s)=\sum_{m\neq0}r_{m,b}(s)e^{2\pi\mathrm{i}mx},
\]
and using
$\delta_0(x)-1=\sum_{m\neq0}e^{2\pi\mathrm{i}mx}$,
we obtain
\[
 r_{m,b}''-(2\pi mb)^2r_{m,b}
 =
 \kappa(\delta_{1/4}+\delta_{-1/4}).
\]
The periodic Green function of
$\frac{d^2}{ds^2}-\lambda^2$ is
\[
 -\frac{1}{2\lambda}
 \sum_{n\in\mathbb Z}e^{-\lambda|s-n|}.
\]
Hence, for $m\neq0$,
\[
 r_{m,b}(s)
 =
 -\frac{\kappa}{4\pi|m|b}
 \sum_{\sigma\in\{-1/4,1/4\}}
 \sum_{n\in\mathbb Z}
 e^{-2\pi|m|b|s-\sigma-n|}.
\]
Since $r_{m,b}=r_{-m,b}$ and $e^{2\pi i mx}+e^{-2\pi i mx}=2\cos(2\pi mx) $, and using
\[
 -\sum_{m=1}^{\infty}\frac{r^m}{m}\cos(m\theta)
 =\log|1-re^{\mathrm{i}\theta}|,
\]
gives
\begin{equation}\label{eq:Sb-Fourier}
 S_b(x,s)-S_0(s)
 =
 \frac{\kappa}{2\pi b}
 \sum_{\sigma\in\{-1/4,1/4\}}
 \sum_{n\in\mathbb Z}
 \log\left|
 1-e^{-2\pi b|s-\sigma-n|+2\pi\mathrm{i}x}
 \right|.
\end{equation}
Summing geometric series shows that
\[
 \sup_{s\in\mathbb R}
 \sum_{\sigma\in\{-1/4,1/4\}}\sum_{n\in\mathbb Z}
 e^{-2\pi b|s-\sigma-n|}\leq C,
 \qquad b\geq1.
\]
Since $\log|1-re^{\mathrm{i}\theta}|\leq r$, it follows that
$S_b\leq S_0+C\kappa/b$, and hence $0\leq h_b\leq C_\kappa$.
For every fixed $s\notin\{-1/4,1/4\}+\mathbb Z$, all the
numbers $e^{-2\pi b|s-\sigma-n|}$ are at most $1/2$ when
$b$ is sufficiently large. Using
$|\log|1-re^{\mathrm{i}\theta}||\leq2r$ for $r\leq1/2$
and the same geometric-series bound in \eqref{eq:Sb-Fourier},
we obtain $S_b(x,s)-S_0(s)\to0$ as $b\to+\infty$.
Thus $h_b\to e^{S_0}$ almost everywhere on $Q$.
The uniform upper bound and dominated convergence give
\begin{equation}\label{eq:hb-convergence}
 0\leq h_b\leq C_\kappa,
 \qquad h_b\longrightarrow h_0:=e^{S_0}
 \quad\text{in }L^q(Q),\quad 1\leq q<\infty.
\end{equation}

For a function on $E_\tau$, we use the same notation for
its pullback to $Q$, and write $\sigma_\tau$ also for
$(x,s)\mapsto(x,s+1/2)$. The rescaled functional is
\[
 J_b(v)=\frac12E_b(v)-2\kappa\log\int_Qh_be^v\,\ud x\ud s,
 \qquad
 E_b(v):=\int_Q(b^2v_x^2+v_s^2)\,\ud x\ud s,
\]
on the space $\int_Qv\,\ud x\ud s=0$.
Since $J_b(v)=bJ(v)+2\kappa\log b$, its minimizers
correspond exactly to those of $J$.

If $\kappa>2\pi^2$, then, for sufficiently large $b$,
\[
 \frac12\lam(E_\tau)|E_\tau|
 =\frac{2\pi^2}{b}<\rho=\frac{\kappa}{b}<4\pi,
\]
and the conclusion follows from Theorem~\ref{thm:nonunique}.
Thus it remains to consider $16<\kappa\leq2\pi^2$.

Suppose, towards a contradiction, that along a sequence
$b\to+\infty$ there are $\sigma_\tau$-invariant global
minimizers $v_b$ satisfying
\begin{align}\label{eq: E-L equation}
   -\left(b^2\partial_x^2+\partial_s^2\right)v_b= e^{U_b}-2\kappa.
\end{align}
Put $I=(-1/4,1/4)$.
Since $v_b(x,s+1/2)=v_b(x,s)$ and $\int_Qv_b=0$,
its restriction to $Q_{1/2}:=(\mathbb R/\mathbb Z)\times I$
has zero average. Applying the Moser--Trudinger inequality on this fixed
flat torus $Q_{1/2}$, 
\begin{align*}
    \log\int_{Q_{1/2}}e^{v_b}\,\ud x\ud s
 \leq C+\frac1{16\pi}\int_{Q_{1/2}}(v_{b,x}^2+v_{b,s}^2)\,\ud x\ud s.
\end{align*} So, form the symmetry of $v$, we can get
\[
 \log\int_Qe^{v_b}\,\ud x\ud s
 \leq C+\frac1{32\pi}\int_Q(v_{b,x}^2+v_{b,s}^2)\,\ud x\ud s
 \leq C+\frac1{32\pi}E_b(v_b).
\]
By \eqref{eq:hb-convergence}, $h_b\leq C_\kappa$.
Also, Jensen's inequality and $\int_QS_b=0$ give
$\int_Qh_b\geq1$, and hence $J_b(v_b)\leq J_b(0)\leq0$.
Consequently,
\[
 0\geq J_b(v_b)
 \geq\left(\frac12-\frac{\kappa}{16\pi}\right)E_b(v_b)-C_\kappa.
\]
Since $\kappa\leq2\pi^2<8\pi$, we obtain
\begin{equation}\label{eq:uniform-energy-bound}
\int_Q(b^2v_x^2+v_s^2)\,\ud x\ud s= E_b(v_b)\leq C_\kappa.
\end{equation}

After passing to a subsequence, compactness on $Q$ gives
\[
 v_b\rightharpoonup v_0(s)\quad\text{in }H^1(Q),
 \qquad
 v_b\longrightarrow v_0(s)\quad\text{in }L^2(Q).
\]
The limit is independent of $x$, since
$\int_Qv_{b,x}^2\leq C_\kappa/b^2$, and satisfies
$v_0(s+1/2)=v_0(s)$.
Applying the Moser--Trudinger inequality to $3v_b$
gives a uniform $L^3(Q)$ bound for $e^{v_b}$.
Thus uniform integrability yields
$e^{v_b}\to e^{v_0}$ in $L^2(Q)$, and
\eqref{eq:hb-convergence} gives
$h_be^{v_b}\to h_0e^{v_0}$ in $L^1(Q)$.

Define
\[
 U_b=S_b+v_b+
 \log\frac{2\kappa}{\int_Qh_be^{v_b}\,\ud x\ud s},
 \qquad
 U_0=S_0+v_0+
 \log\frac{2\kappa}{\int_Qh_0e^{v_0}\,\ud x\ud s}.
\]
Then $u(x+\tau s)=U_b(x,s)-2\log b$ is the
corresponding solution on $E_\tau$, and
\begin{equation}\label{eq:Ub-convergence}
 e^{U_b}\longrightarrow e^{U_0}\quad\text{in }L^1(Q),
 \qquad \int_Qe^{U_0}\,\ud x\ud s=2\kappa.
\end{equation}
Passing to the Euler--Lagrange equation \eqref{eq: E-L equation} against tests
depending only on $s$ gives $-v_0''=e^{U_0}-2\kappa$.
Together with the equation for $S_0$ and half-periodicity,
this yields
\begin{equation}\label{eq:limit-cell}
 U_0''+e^{U_0}=0\quad\text{on }I,
 \qquad U_0(-1/4)=U_0(1/4),
 \qquad \int_Ie^{U_0}\,\ud s=\kappa.
\end{equation}
In particular, $U_0\in W^{1,\infty}([-1/2,1/2])$ and  is smooth up to the endpoints from
within $I$.

We now construct a negative direction for the second variation.
Since $(U_0')^2+2e^{U_0}$ is constant on $I$,
the equal endpoint values and the mass in
\eqref{eq:limit-cell} give
\[
 U_0'(-1/4)=\frac{\kappa}{2},
 \qquad U_0'(1/4)=-\frac{\kappa}{2}.
\]
The above derivatives are understood as one-sided derivatives taken from the interior of $I$.
 Set
\[
 \psi(s)=\frac{\kappa}{16}+\frac{s}{2}U_0'(s).
\]
Then $\psi\in H^1_0(I)$. Moreover, $U_0''=-e^{U_0}<0$
implies $|U_0'|\leq\kappa/2$, so
$\psi(s)\geq\kappa(1-4|s|)/16>0$ in $I$.
Differentiating the Liouville equation gives
\[
 -\psi''-e^{U_0}\psi
 =\left(1-\frac{\kappa}{16}\right)e^{U_0}
 =-\varepsilon e^{U_0}.
\]
Extend $\psi$ to a $1$-periodic function $\varphi$ by
$\varphi=\psi$ on $I$ and
$\varphi(s+1/2)=-\varphi(s)$.
Then $\varphi\in H^1(Q)\cap L^\infty(Q)$ is independent
of $x$, and half-periodicity gives
$\int_Q\varphi=\int_Qe^{U_b}\varphi=0$.
The second-variation formula in the proof of
Theorem~\ref{thm:nonunique} therefore gives
\[
 0\leq D^2(J_b)_{v_b}(\varphi)
 =2\int_I\psi'^2\,\ud s
  -\int_Qe^{U_b}\varphi^2\,\ud x\ud s.
\]
Using \eqref{eq:Ub-convergence} and integrating the equation
for $\psi$ against $\psi$, we obtain
\begin{equation}\label{eq:negative-limit}
 0\leq2\int_I(\psi'^2-e^{U_0}\psi^2)\,\ud s
 =-2\varepsilon\int_Ie^{U_0}\psi\,\ud s<0,
\end{equation}
a contradiction. Thus, for sufficiently large $b$, a global minimizer and
its $\sigma_\tau$-translate give two distinct solutions,
neither of which is $\sigma_\tau$-invariant.
\end{proof}
\begin{remark}\label{rem:asymptotic-sharpness}
For $\tau=\mathrm{i}b$ with $b\geq1$,
$
	\lam(E_\tau)|E_\tau|
	=
	\frac{4\pi^2}{b}.
$
Therefore the examples in
Theorem~\ref{thm:asymptotic-nonunique} satisfy
\[
	\rho
	=
	\frac{4(1+\varepsilon)}{\pi^2}
	\lam(E_\tau)|E_\tau|.
\]
Since the total singular mass of
\eqref{eq:asymptotic-nonunique} is $2\rho$, equivalently,
\[
	2\rho
	=
	\frac{8(1+\varepsilon)}{\pi^2}
	\lam(E_\tau)|E_\tau|.
\]
Hence the coefficient $8/\pi^2$ in the asymptotic uniqueness
estimate for the total singular mass is optimal.
\end{remark}

Next we continue to study the uniqueness and non-uniqueness of solutions for 
\begin{equation}
\Delta u+e^{u}=\rho(\delta_{p}+\delta_{-p})\quad\text{ on
}\; E_{\tau} \label{equ1-3}%
\end{equation}
when the real parameter $\rho$ is close to $4\pi$. To this goal, we need to recall some basic facts about the Green function
$$G_p(z)=\frac12\big(G(z-p)+G(z+p)\big),$$
where $p\in E_{\tau}\setminus E_{\tau}[2].$ Clearly $G_p(z)$ is even, so $-a$ is also a critical point of $G_p(z)$ if $a\in E_{\tau}\setminus\{0\}$ is. Clearly $\frac{\omega_k}{2}$, $k\in \{0,1,2,3\}$, are always critical points of $G_p(z)$. 

\medskip

\noindent{\bf Definition.} {\it A critical point $a\in E_{\tau}$ of $G_p$ is called trivial if $a=-a$ in $E_{\tau}$, i.e., $a\in E_{\tau}[2]$.  A critical point $a\in E_{\tau}$ is called nontrivial if $a\neq-a$ in $E_{\tau}$, i.e., $a\notin E_{\tau}[2]$.}
\medskip

Therefore, $\frac{\omega_k}{2}$, $k=0,1,2,3$, are all trivial critical points of $G_p(z)$, i.e., the number of critical points of $G_p(z)$ is an even number of at least $4$.

\begin{theorem}\cite{CFL}\label{main-thm-1} For any $p\in E_{\tau}\setminus E_{\tau}[2]$, $G_p(z)$ has at most $3$ pairs of nontrivial critical points, or equivalently, the number of critical points of $G_p(z)$ belongs to $\{4,6,8,10\}$, and each number in $\{4,6,8,10\}$ really occurs for different $(\tau, p)$'s. 

Moreover,  fix any $\tau$, then for almost all $p\in E_{\tau}\setminus E_{\tau}[2]$, all critical points of $G_p(z)$ are non-degenerate.
\end{theorem}

The deep connection between \eqref{equ1-3} and $G_p(z)$ is realized via the bubbling
phenomena. Let $u_{k}$ be a sequence of solutions of (\ref{equ1-3}) with
$\rho=\rho_{k}\rightarrow4\pi$, and $\max_{E_{\tau}}%
u_{k}(z)\rightarrow+\infty$ as $k\rightarrow+\infty$. We call $q$ a blowup point of $\{u_k\}$ if there is a sequence $\{x_k\}_k$ such that $x_k\to q$ and $u_k(x_k)\to+\infty$ as $k\to+\infty$. Then it follows from
\cite{BT,CL-1} that $u_{k}$ has exactly one blowup point $q\in E_{\tau}\setminus\{\pm p\}$. Furthermore,
\[u_k(z)+2\rho_kG_p(z)-\frac{1}{|E_{\tau}|}\int_{E_{\tau}}u_k=\int_{E_\tau}G(z-y)e^{u_k(y)}dy\to 8\pi G(z-q)\]
uniformly in $K\Subset E_{\tau}\setminus\{q\}$. From here,
the
well-known \textit{Pohozaev identity} says that $\nabla G_p(q)=0$, i.e., the blowup point $q$ is a critical point of $G_p(z)$. 

There is an important quantity $D(q)$ related to the bubbling phenomenon.
Define the regular part $\tilde{G}(z,w)$ of $G(z,w)$ by
\begin{equation}\label{equ1-4}
\tilde{G}(z,w):=G(z,w)+\frac{1}{2\pi}\log|z-w|.
\end{equation}
Define
\begin{align}\label{equ1-5}
D(q):=\lim_{r\to0}\bigg(\int_{E_{\tau}\setminus B_r(q)}\frac{e^{8\pi (\tilde{G}(z,q)-\tilde{G}(q,q))-8\pi G_p(z)}-e^{-8\pi G_p(q)}}{|z-q|^4}-\int_{\mathbb R^2\setminus E_{\tau}}\frac{e^{-8\pi G_p(q)}}{|z-q|^4}
\bigg).\end{align}

\begin{theorem}\cite[Theorem 3.1 and Lemma 3.2]{CLW2004}\label{thm-B}
Let $u_{k}$ be a
sequence of bubbling solutions of \eqref{equ1-3} with $\rho=\rho
_{k}\rightarrow4\pi$, and denote by $q$ the blowup point, which is a critical point of
$G_{p}$. Denote $\lambda_{k}:=\max_{E_{\tau}}u_{k}(z)$, then there is a constant $c_q>0$ such that
\begin{equation}\label{bubb}
\rho_{k}-4\pi=c_q(D(q)+o(1))e^{-\lambda_{k}}.
\end{equation}
\end{theorem}
Therefore, when $D(q)\neq 0$, then $D(q)$ controls the sign of $\rho_k-4\pi$. For example, if $q$ is a critical point of $G_p(z)$ satisfying $D(q)>0$, then \eqref{bubb} implies $\rho_k>4\pi$, which means that $u_k$ can not blow up at this $q$ if $u_k$ is a sequence of solutions of \eqref{equ1-3} with $\rho=\rho
_{k}\uparrow 4\pi$.

\begin{theorem}\cite{CFL}\label{thm-CFL2}
Let $p\in E_{\tau}\setminus E_{\tau}[2]$. 
\begin{itemize}
\item[(1)] Once \begin{equation}
\Delta u+e^{u}=4\pi(\delta_{p}+\delta_{-p})\quad\text{ on
}\; E_{\tau}, \label{mean}%
\end{equation} has a solution, then it has a one-parameter scaling family of solutions $u_{\beta}(z)$, where $\beta>0$ is arbitrary.
\item[(2)]
There is a one-to-one correspondence between pairs of nontrivial critical points of $G_p(z)$ and one-parameter scaling families of solutions of \eqref{mean}. 
\item[(3)]
 \eqref{mean} has at most $3$ one-parameter scaling families of solutions, and every number in $\{0,1,2,3\}$ really occurs. 
\end{itemize}
\end{theorem}

\begin{remark}
Let $q$ be a nontrivial critical point of $G_p(z)$. Then $D(q)=0$.
Indeed, it follows from \cite{CFL} that \eqref{mean}
has a one-parameter scaling family of solutions $u_\beta(z)$ such that $u_\beta(z)$ blows up at $q$ as $\beta\to+\infty$. Then by \eqref{bubb} we obtain $D(q)=0$.

Therefore, if $D(q)\neq 0$ for some critical point of $G_p(z)$, then $q$ must be a trivial critical point of $G_p(z)$, i.e., $q=\frac{\omega_k}{2}$ for some $k=0,1,2,3$.
\end{remark}

The following result is a special case of \cite[Theorem 1.5]{EF}.
\begin{theorem}\cite[Theorem 1.5]{EF}\label{thm-C}
Suppose $q
=\frac{\omega_k}{2}$ is a trivial critical point of $G_p(z)$ such that both
$D(q)$ and $\det D^{2}G_{p}(q)$ do not vanish. Then
there exists $\varepsilon_{0}>0$ such that for any $0<\varepsilon
<\varepsilon_{0}$, equation \eqref{equ1-3} with $$\rho=\begin{cases}4\pi+\varepsilon\quad\text{if }D(q)>0\\
4\pi-\varepsilon\quad\text{if }D(q)<0\end{cases}$$ possesses a solution $u_{\varepsilon}(z)$. Moreover, $u_{\varepsilon
}(z)$ blows up exactly at $q$ as $\varepsilon
\rightarrow0$.
\end{theorem}

\begin{theorem}\cite{BKLY, BYZ}\label{thm-D}
Suppose $q
=\frac{\omega_k}{2}$ is a trivial critical point of $G_p(z)$ such that both
$D(q)$ and $\det D^{2}G_{p}(q)$ do not vanish.
Suppose $u_{k}(z)$, $\tilde{u}_{k}(z)$ are two sequences of
solutions of \eqref{equ1-3} with the same parameter $\rho
_{k}\rightarrow4\pi$, bot of which  blow up
at $q$. Then $u_{k}(z)=\tilde
{u}_{k}(z)$ and $u_k(z)$ is non-degenerate for large $k$.
\end{theorem}
In Theorem \ref{thm-D}, the uniqueness $u_{k}(z)=\tilde
{u}_{k}(z)$ was proved in \cite{BKLY}, and the nondegeneracy was proved in \cite{BYZ}.

In view of Theorems \ref{thm-C}-\ref{thm-D}, it is important to see when $$D(q)\neq 0\quad\text{and}\quad\det D^2G_p(q)\neq 0$$ for a trivial critical point $q=\frac{\omega_k}{2}$. The following result shows that these two quantities can be related to each other.

\begin{lemma}\label{lemma1-11}
Let  $q=\frac{\omega_k}{2}$ be a trivial critical point of $G_p(z)$. Write $b=\operatorname{Im} \tau>0$. Then
\begin{equation}\label{equ1-10}
D(q)=-4\pi^2be^{-8\pi G_p(q)} \det D^2G_p(q).
\end{equation}
\end{lemma}

\begin{proof}
By \eqref{equ1-4}-\eqref{equ1-5}, we have
\begin{align}\label{equ1-6}
D(q)=e^{-8\pi G_p(q)}\lim_{r\to0}\bigg(\int_{E_{\tau}\setminus B_r(q)}e^{8\pi (G_p(q)-\tilde{G}(q,q))}e^{8\pi(G(z,q)-G_p(z))}\nonumber\\
-\int_{\mathbb R^2\setminus B_r(q)}\frac{1}{|z-q|^4}
\bigg).\end{align}
Note that $8\pi(G(z,q)-G_p(z))$ is a doubly periodic harmonic function in $\mathbb R^2$ with singularities $-4\log|z-q|$ at $z=q$ and $2\log |z\pm p|$ at $z=\pm p$. Thanks to $q=\frac{\omega_k}{2}$, we obtain
$$8\pi(G(z,q)-G_p(z))=2\log|\wp(z-q)-\wp(p-q)|+C,$$ 
where the constant $C=8\pi (\tilde{G}(q,q)-G_p(q))$. Remark that this formula does not hold if $q\notin E_{\tau}[2]$. Inserting this formula into \eqref{equ1-6}, we obtain
\begin{align}\label{equ1-7}
D(q)=&e^{-8\pi G_p(q)}\lim_{r\to0}\bigg(\int_{E_{\tau}\setminus B_r(q)}|\wp(z-q)-\wp(p-q)|^2-\int_{\mathbb R^2\setminus B_r(q)}\frac{1}{|z-q|^4}
\bigg)\nonumber\\
=&e^{-8\pi G_p(q)}\lim_{r\to0}\bigg(\int_{E_{\tau}\setminus B_r(0)}|\wp(z)-\wp(p-q)|^2-\int_{\mathbb R^2\setminus B_r(0)}\frac{1}{|z|^4}
\bigg)\nonumber\\
=&e^{-8\pi G_p(q)}\lim_{r\to0}\bigg(\int_{E_{\tau}\setminus B_r(0)}|\wp(z)|^2-\int_{\mathbb R^2\setminus B_r(0)}\frac{1}{|z|^4}
\bigg)\\
&-e^{-8\pi G_p(q)}\bigg(\lim_{r\to0}\int_{E_{\tau}\setminus B_r(0)}(\overline{\wp(p-q)}\wp(z)+\wp(p-q)\overline{\wp(z)})-b|\wp(p-q)|^2\bigg).\nonumber\end{align}
It was computed in \cite[Theorem 3.2]{LW4} that
$$\lim_{r\to0}\bigg(\int_{E_{\tau}\setminus B_r(0)}|\wp(z)|^2-\int_{\mathbb R^2\setminus B_r(0)}\frac{1}{|z|^4}
\bigg)=b|\eta_1|^2-\pi (\eta_1+\bar\eta_1),$$
$$\lim_{r\to0}\int_{E_{\tau}\setminus B_r(0)}\wp(z)=\pi-\eta_1b,$$
$$\lim_{r\to0}\int_{E_{\tau}\setminus B_r(0)}\overline{\wp(z)}=\pi-\bar\eta_1b.$$
Here, $\eta_1$ is the first quasi-period of the Weierstrass zeta function.
Inserting these formulas into \eqref{equ1-7}, we obtain
\begin{align*}
D(q)=&e^{-8\pi G_p(q)}\bigg(b|\eta_1|^2-\pi (\eta_1+\bar\eta_1)+\overline{\wp(p-q)}(\eta_1b-\pi)\\
&+\wp(p-q)(\bar\eta_1b-\pi)+b|\wp(p-q)|^2\bigg)\\
=&e^{-8\pi G_p(q)}\bigg(b|\wp(p-q)+\eta_1|^2-\pi(\wp(p-q)+\eta_1+\overline{\wp(p-q)+\eta_1})\bigg)\\
=&be^{-8\pi G_p(q)}\bigg(\bigg|\wp(p-q)+\eta_1-\frac{\pi}{b}\bigg|^2-\frac{\pi^2}{b^2}\bigg).
\end{align*}
Since it was proved in \cite[(3.2)]{CFL} that
$$\det D^2G_p(q)=\frac{-1}{4\pi^2}\bigg(\bigg|\wp(p-q)+\eta_1-\frac{\pi}{b}\bigg|^2-\frac{\pi^2}{b^2}\bigg),$$
we obtain \eqref{equ1-10}.
\end{proof}

Remark that when $\rho\notin 4\pi\mathbb{N}$, the Leray-Schauder degree $d_\rho$ of \eqref{equ1-3} is well-defined. Furthermore, it follows from \cite[Theorem 1.1]{CL-3} that
\begin{equation}
d_\rho=\begin{cases} 1\quad\text{if }\rho\in (0, 4\pi)\\
3\quad\text{if }\rho\in (4\pi, 8\pi).
\end{cases}
\end{equation}
In particular, \eqref{equ1-3} with $\rho\in (0,4\pi)\cup(4\pi,8\pi)$ always has solutions.

\begin{corollary}\label{coro1-3}
Suppose $q=\frac{\omega_k}{2}$ is a non-degenerate trivial critical point of $G_p(z)$.
\begin{itemize}
\item[(1)]If $\det D^2G_p(q)>0$ or equivalently $D(q)<0$, then there exists $\varepsilon>0$ such that for any $\rho\in (4\pi-\varepsilon, 4\pi)$, equation \eqref{equ1-3} has exactly one solution $u_\varepsilon$ such that $u_\varepsilon$ blows up at $q$ as $\varepsilon\to 0$. Furthermore, $u_\varepsilon(z)=u_\varepsilon(-z)$, $u_\varepsilon$ is non-degenerate and hence contributes degree $\pm 1$. Besides, \eqref{equ1-3} with $\rho>4\pi$ has no solutions which blows up at $q$ as $\rho\downarrow 4\pi$.
\item[(2)]If $\det D^2G_p(q)<0$ or equivalently $D(q)>0$, then there exists $\varepsilon>0$ such that for any $\rho\in (4\pi, 4\pi+\varepsilon)$, equation \eqref{equ1-3} has exactly one solution $u_\varepsilon$ such that $u_\varepsilon$ blows up at $q$ as $\varepsilon\to 0$. Furthermore, $u_\varepsilon(z)=u_\varepsilon(-z)$, $u_\varepsilon$ is non-degenerate and hence contrubutes degree $\pm 1$. Besides, \eqref{equ1-3} with $\rho<4\pi$ has no solutions which blows up at $q$ as $\rho\uparrow 4\pi$.
\end{itemize}
\end{corollary}

\begin{proof}
This is a direct consequence of Theorems \ref{thm-B}-\ref{thm-D} and Lemma \ref{lemma1-11}. Note that $u_\varepsilon(z)=u_\varepsilon(-z)$ follows from Theorem \ref{thm-D} and the fact that $u_{\varepsilon}(-z)$ is also a solution of \eqref{equ1-3} and blows up at $q$ because $-q=q$.
\end{proof}

Define
\begin{align}\label{mnumber}
m^-:=  \# \Big\{q=\frac{\omega_k}{2}\,|\, D({q}%
)<0\Big\} 
=  \# \Big\{ q=\frac{\omega_k}{2}\,|\, \det
D^{2}G_{p}(q)>0\Big\},\\
m^+:=  \# \Big\{q=\frac{\omega_k}{2}\,|\, D({q}%
)>0\Big\} 
=  \# \Big\{ q=\frac{\omega_k}{2}\,|\, \det
D^{2}G_{p}(q)<0\Big\}.
\end{align}

\begin{theorem}Fix $\tau$ and $p\in E_{\tau}\setminus E_{\tau}[2]$ such that $G_p(z)$ has no nontrivial critical points and the $4$ trivial critical points are all non-degenerate (It follows from \cite{CFL} that such $(\tau, p)$ exists). Then there exists small $\varepsilon>0$ such that
\begin{itemize}
\item[(1)] For $\rho\in (4\pi-\varepsilon, 4\pi)$, \eqref{equ1-3} has a unique solution that is even.

\item[(2)]  For $\rho\in (4\pi, 4\pi+\varepsilon)$, \eqref{equ1-3} has exactly three solutions that are all even.
\end{itemize}

\end{theorem}
\begin{proof}
Since $G_p(z)$ has no nontrivial critical points, it follows from Theorem \ref{thm-CFL2} that \eqref{equ1-3} with $\rho=4\pi$ has no solutions. Therefore, if $u_\rho$ is any solution of \eqref{equ1-3} with $\rho\neq 4\pi$, then $u_\rho$ must blow up and then blows up at some trivial critical point of $G_p(z)$ as $\rho\to 4\pi$. Furthermore, since  $G_p(z)$ has no nontrivial critical points and the $4$ trivial critical points are all non-degenerate, it follows from \cite{CFL} that $m^-=1$ and $m^+=3$. Therefore, the assertions (1)-(2) follow directly from Corollary \ref{coro1-3}.
\end{proof}

\begin{theorem}Fix $\tau$ and $p\in E_{\tau}\setminus E_{\tau}[2]$ such that $G_p(z)$ has nontrivial critical points, the $4$ trivial critical points are all non-degenerate and $m^-=m^+=2$ (It follows from \cite{CFL} that such $(\tau, p)$ exists). Then there exists small $\varepsilon>0$ such that
\begin{itemize}
\item[(1)] For $\rho\in (4\pi-\varepsilon, 4\pi)$, \eqref{equ1-3} has at least $3$ even solutions.

\item[(2)]  For $\rho\in (4\pi, 4\pi+\varepsilon)$, \eqref{equ1-3} has at least $3$ even solutions.
\end{itemize}
\end{theorem}
\begin{proof} Let $k_1, k_2\in \{0,1,2,3\}$ such that $D(\frac{\omega_{k_1}}{2})<0$ and $D(\frac{\omega_{k_2}}{2})<0$. Then by Corollary \ref{coro1-3}, there is small $\varepsilon>0$ such that for $\rho\in (4\pi-\varepsilon, 4\pi)$, \eqref{equ1-3} has two non-degenerate even solutions $u_{\rho,1}$ and $u_{\rho,2}$ such that $u_{\rho,i}$ blows up at $\frac{\omega_{k_i}}{2}$ for $i=1,2$. Since the total degree of these two solutions can only be $-2, 0$ or $2$, it follows from the Leray-Schauder degree $d_\rho=1$ that for $\rho\in (4\pi-\varepsilon, 4\pi)$, \eqref{equ1-3} must have other solutions. If all other solutions are not even, since $u(-z)$ is also a solution if $u(z)$ is, it follows from \cite[Corollary 2.1]{Wang} for $\mathbb Z_2$-symmetry that the total degree of all other solutions must be even, a contradiction with $d_\rho=1$. Therefore, we conclude that \eqref{equ1-3} must have the third even solution. This proves the assertion (1). The assertion (2) can be proved similarly.
\end{proof}

\begin{theorem}Fix $\tau$ and $p\in E_{\tau}\setminus E_{\tau}[2]$ such that $G_p(z)$ has nontrivial critical points, the $4$ trivial critical points are all non-degenerate and $m^-=0$, $m^+=4$ (It follows from \cite{CFL} that such $(\tau, p)$ exists). Then there exists small $\varepsilon>0$ such that
\begin{itemize}
\item[(1)] For $\rho\in (4\pi-\varepsilon, 4\pi)$, \eqref{equ1-3} has at least one even solution.

\item[(2)]  For $\rho\in (4\pi, 4\pi+\varepsilon)$, \eqref{equ1-3} has at least $5$ even solutions.
\end{itemize}
\end{theorem}
\begin{proof} 
The assertion (1) is trivial because $d_\rho=1$. 
To prove the assertion (2), we see from  $m^+=4$ and Corollary \ref{coro1-3} that there is small $\varepsilon>0$ such that for $\rho\in (4\pi, 4\pi+\varepsilon)$, \eqref{equ1-3} has four non-degenerate even solutions $u_{\rho,k}$ such that $u_{\rho,k}$ blows up at $\frac{\omega_{k-1}}{2}$ for $k=1,2,3,4$. Since the total degree of these four solutions can only be an even number, and the Leray-Schauder degree $d_\rho=3$, the same argument as above implies that for $\rho\in (4\pi, 4\pi+\varepsilon)$, \eqref{equ1-3} must have another even solution $u_{\rho,5}$. This proves the assertion (2).
\end{proof}

\subsection*{Acknowledgments}  Z. Chen is supported by National Key R\&D Program of China (No. 2023YFA1010002) and NSFC (No. 12222109). 
S. Zhang is supported by the Postdoctoral Fellowship Program and China Postdoctoral Science Foundation under Grant Numbers BX20250062 and 2026M793381.
The authors acknowledge the use of AI tools. The authors assume full responsibility for the
mathematical validity, accuracy, and integrity of the proofs presented in the manuscript.

\end{document}